\documentclass[twoside,11pt]{preprintCL}

\usepackage[T1]{fontenc}
\usepackage[english]{babel}
\usepackage{times}
\usepackage{amsmath}
\usepackage{amsfonts}
\usepackage{amssymb}
\usepackage{amsmath}
\usepackage{amsthm}
\usepackage{graphicx}
\usepackage{array}
\usepackage{color}
\usepackage{mathrsfs}
\usepackage{hyperref}
\usepackage{esint} 
\usepackage{tikz}
\usepackage{ upgreek }
\usepackage{enumitem}

\definecolor{darkgreen}{rgb}{0.1,0.7,0.1}

\definecolor{airforceblue}{rgb}{0.36, 0.54, 0.66}

\newtheorem{theorem}{Theorem}
\newtheorem{lemma}{Lemma}[section]
\newtheorem{proposition}[lemma]{Proposition}

\newtheorem{conjecture}[lemma]{Conjecture}
\newtheorem{remark}[lemma]{Remark}
\newtheorem{definition}[lemma]{Definition}

\newcommand{\bn}[1]{{[\kern-0.5ex] #1 
    [\kern-0.5ex]}}

\newcommand{\eps}{\varepsilon}

\newcommand\symb[2][\bf]{{\mathchoice{\hbox{#1#2}}{\hbox{#1#2}}%
        {\hbox{\scriptsize#1#2}}{\hbox{\tiny#1#2}}}}
    
\newcommand{\crochet}[1]{\left\langle #1 \right\rangle}
\newcommand{\norm}[1]{\left\| #1 \right\|}
\newcommand{\normm}[1]{{\left\vert\!\left\vert\!\left\vert #1 \right\vert\!\right\vert\!\right\vert}}

\def\R{{\symb R}}
\def\N{{\symb N}}
\def\Z{{\symb Z}}

\def\P{{\symb P}}

\def\un{\mathbf{1}}

\def\rm{\mathrm{m}}

\def\${|\!|\!|}

\newcommand{\cc}{\complement}

\renewcommand{\P}{\mathbb{P}}
\newcommand{\E}{\mathbb{E}}

\newcommand{\cA}{\mathcal{A}}

\newcommand{\cC}{\mathcal{C}}

\newcommand{\cF}{\mathcal{F}}
\newcommand{\cG}{\mathcal{G}}
\newcommand{\cH}{\mathcal{H}}
\newcommand{\cI}{\mathcal{I}}

\newcommand{\cM}{\mathcal{M}}

\newcommand{\cT}{\mathcal{T}}
\newcommand{\cU}{\mathcal{U}}

\newcommand{\ccB}{\mathscr{B}}

\newcommand{\ccD}{\mathscr{D}}

\hypersetup{
citecolor=blue,
colorlinks=true, 
breaklinks=true, 
urlcolor= blue, 
linkcolor= black, 
bookmarksopen=true, 
pdftitle={Asymptotics of Anderson Hamiltonian}, 
}

\begin{document}

\title{Asymptotic of the smallest eigenvalues of the\\ continuous Anderson Hamiltonian in $d\le 3$}
\author{Yueh-Sheng Hsu\footnote{Universit\'e Paris-Dauphine, PSL University, CNRS, UMR 7534, CEREMADE, 75016 Paris, France. \email{hsu@ceremade.dauphine.fr}}\; and Cyril Labb\'e\footnote{Universit\'e de Paris, Laboratoire de Probabilit\'es, Statistiques et Mod\'elisation, UMR 8001, F-75205 Paris, France \email{clabbe@lpsm.paris}}}

\vspace{2mm}

\date{\today}
\maketitle

\begin{abstract}
We consider the continuous Anderson Hamiltonian with white noise potential on $(-L/2,L/2)^d$ in dimension $d\le 3$, and derive the asymptotic of the smallest eigenvalues when $L$ goes to infinity. We show that these eigenvalues go to $-\infty$ at speed $(\log L)^{1/(2-d/2)}$ and identify the prefactor in terms of the optimal constant of the Gagliardo-Nirenberg inequality. This result was already known in dimensions $1$ and $2$, but appears to be new in dimension $3$. We present some conjectures on the fluctuations of the eigenvalues and on the asymptotic shape of the corresponding eigenfunctions near their localisation centers.
	
	\medskip
	
\end{abstract}

\setcounter{tocdepth}{2}
\tableofcontents

\section{Introduction}

Given a white noise $\xi$ on $\R^d$, we consider the truncated continuous Anderson Hamiltonian
$$ \cH_L := -\Delta + \xi\;,\quad \mbox{ on } Q_L := (-L/2,L/2)^d\;,$$
where $\Delta$ is the continuous Laplacian, boundary conditions are taken to be homogeneous Dirichlet and the dimension $d$ is either $1$, $2$ or $3$.\\

This operator belongs to the class of random Schr\"odinger operators. The particularity of the present setting is the singularity of the white noise potential, which is only distribution-valued. In dimension $1$, the operator $\cH_L$ can be defined with standard tools and rather complete results are available on the asymptotic behaviour as $L\to\infty$ of its eigenvalues and eigenfunctions, see~\cite{McKean,DL19,DLCritical,DLCrossover}.\\

On the other hand, the mere definition of the operator in dimensions $2$ and $3$ is a priori unclear. Indeed, the regularity of white noise is too low for the operator to be defined by classical arguments, and it actually needs to be \emph{renormalised} by infinite constants. New techniques in the field of stochastic PDEs have provided the appropriate tools to carry out such a construction. Building on the paracontrolled calculus of Gubinelli, Imkeller and Perkowski~\cite{GIP15}, Allez and Chouk~\cite{AllezChouk} constructed $\cH_L$ in dimension $2$ under periodic b.c. This construction was extended to dimension $3$ under periodic b.c.~in~\cite{GUZ}, under Dirichlet b.c.~in dimension $2$ in~\cite{CvZ19} and to $2$-dimensional manifolds in~\cite{Mouzard}. On the other hand, a construction under periodic and Dirichlet b.c.~and for any dimension $d\le 3$ was presented in~\cite{Lab19} using the theory of regularity structures~\cite{Hai14}: in the present article we rely on this construction for convenience.

Let us provide a brief description of the aforementioned renormalisation procedure. Consider the operator $\cH_{\eps,L} = -\Delta + \xi_\eps + C_\eps$ associated with a regularized noise $\xi_\eps = \xi * \rho_\eps$, where $\rho_\eps$ is a smooth function that lives at scale $\eps$. This operator is well-defined since $\xi_\eps$ is a smooth function. In the references above, it is shown that if one chooses properly the \emph{renormalisation constant} $C_\eps$, then $\cH_{\eps,L}$ converges in norm resolvent sense to some limit that we call $\cH_L$. Note that, as $\eps \downarrow 0$, $C_\eps$ diverges logarithmically in dimension $2$ and polynomially in dimension $3$. We refer the reader to Section \ref{sec:LDP} for further details.

In fine, these constructions yield a self-adjoint operator $\cH_L$ on $L^2((-L/2,L/2)^d)$ with pure point spectrum bounded from below: we let $(\lambda_{k,L})_{k\ge 1}$ be its eigenvalues in non-decreasing order and $(\varphi_{k,L})_{k\ge 1}$ be the corresponding eigenfunctions normalised in $L^2$. In contrast with dimension $1$, very little is known on the spectrum of $\cH_L$: in dimension $2$, the asymptotic behaviour as $L\to\infty$ of the smallest eigenvalues was derived in~\cite{CvZ19} while the existence of a density of states was proven in~\cite{Matsuda}; in dimension $3$, essentially no result on the spectrum is available.\\

For later use, let us recall the Gagliardo-Nirenberg inequality - also referred to as Ladyzhenskaya's inequality
\begin{equation}\label{Eq:GN} \|f\|_{L^4(\R^d)} \le C \|\nabla f\|_{L^2(\R^d)}^{d/4} \|f\|_{L^2(\R^d)}^{1-d/4}\;,\end{equation}
and let $\kappa_d$ be the associated optimal constant, that is
\begin{equation}\label{Eq:kappa}
\kappa_d := \sup_{\substack{f \in H^1(\R^d)}} \frac{\|f\|_{L^4(\R^d)}}{\|\nabla f\|_{L^2(\R^d)}^{d/4} \|f\|_{L^2(\R^d)}^{1-d/4}}\;.
\end{equation}
The main result of this article is as follows.

\begin{theorem}
	Fix $d \in \{1, 2, 3\}$ and $n\in \N$. Then almost surely
	\begin{equation}
	\lambda_{n,L} \sim - \left(C_d \log L\right)^{\frac{1}{2 - \frac{d}{2}}}\;,\quad L\to\infty\;.
	\end{equation}
	The constant $C_d$ can be expressed in terms of the Gagliardo-Nirenberg constant through the relation
	\begin{equation}
	C_d = \frac{d^{1 + \frac{d}{2}}(4-d)^{2 - \frac{d}{2}}}{8} \,\kappa_d^{4}\;.
	\label{eq:C_d}
	\end{equation}
	\label{thm_asymp}
\end{theorem}

Let us make some comments on this result. In dimension $1$, the Gagliardo-Nirenberg constant is known to be $\kappa_1 = 3^{-1/8}$ and the result rewrites
$$\lambda_{n,L} \sim -\left(\frac{3}{8} \log L\right)^{2/3}\;,\quad L\to\infty\;.$$
This asymptotic is already covered by more precise results~\cite{McKean,DL19}, in which not only the asymptotic behaviour of $\lambda_{n,L}$ but also its fluctuations are derived, see the end of the introduction for more details. It can also be connected to a result of Chen~\cite{Chen} on the total-mass of the associated parabolic Anderson model.

In higher dimension the Gagliardo-Nirenberg constant $\kappa_d$ is not explicit anymore. In dimension $2$, the asymptotic is
$$\lambda_{n,L} \sim - \kappa_2^4 \log L\;,\quad L\to\infty\;,$$
and was recently established by Chouk and van Zuijlen~\cite{CvZ19} (see also~\cite{Gaudreau} for related results for smooth Gaussian noises). A minor improvement over their result is that our convergence holds almost surely over $L\to\infty$, and not only over sequences $L_k\to\infty$, see Remark \ref{Rk:Op} for some explanations. In dimension $3$, the asymptotic is
$$\lambda_{n,L} \sim -\frac{243}{64} \kappa_3^8 \left(\log L\right)^2\;,\quad L\to\infty\;,$$
and in that case, the result is new.\\

Our proof is carried out simultaneously in all dimensions $d\le 3$ in order to emphasise the dependence on $d$ of the arguments and of the overall result. Let us point out that we follow the same strategy of proof as Chouk and van Zuijlen~\cite{CvZ19} who covered the case of dimension $2$. In fact, the proof essentially boils down to establishing a tail estimate on the principal eigenvalue: this is the content of the next (more general) result.

\begin{theorem}
	Fix $\eta \in (0,1)$ and $n\ge 1$. There exist $\gamma_2 > \gamma_1 > 0$ and $x_0 > 0$ such that the following inequalities hold: for all $L \ge 1$ and all $x \ge x_0$ we have 
	\begin{align}\label{Eq:Tail}
	\exp\left[ - \gamma_2 x^{d/2} e^{d\log L - (1 - \eta) \rho x^{2 - d/2}} \right] \leq
	\P(\lambda_{n, L} \geq -x) \leq \exp\left[ - \gamma_1 x^{d/2} e^{d\log L - (1 + \eta) \rho x^{2 - d/2}}\right]
	\end{align}
	with $\rho = d / C_d$.
	\label{thm:tail_estimates}
\end{theorem}
\noindent Observe that in the limit $L\to\infty$ and for small $\eta$ (take $\eta = 0$ for simplicity), the leftmost and rightmost functions in \eqref{Eq:Tail} pass abruptly from $0$ to $1$ around the critical value $x_c= (C_d \log L)^{1/(2-d/2)}$ where the exponent $d\log L - \rho x^{2 - d/2}$ vanishes. This implies that the distribution function $x\mapsto \P(\lambda_{n, L} \geq -x)$ is close to $0$ for $x \ll x_c$ and close to $1$ for $x\gg x_c$, and therefore that the distribution of $-\lambda_{n, L}$ concentrates near this critical value. Given this result, the derivation of Theorem \ref{thm_asymp} is relatively elementary.

\bigskip

We conclude this introduction with some conjectures. Let $a_L$ be the unique solution to the equation
 \begin{equation*}
 \frac{d}{2} \log a_L + d \log L - \rho a_L^{2 - d/2} = 0\;,
 \end{equation*}
and set
 \begin{equation*}
 	b_L := \frac{C_d}{d (2-\frac{d}{2}) a_L^{1-\frac{d}{2}}}\;.
 \end{equation*}
Note that the asymptotic expansion of $a_L$ is given by
\begin{equation}\label{eq:a_L}
	a_L^{2-d/2} := (C_d\log L) \left[ 1 + \frac{1}{4-d} \frac{\log\log L}{\log L} + o\left(\frac{\log\log L}{\log L}\right)\right].
\end{equation}

\begin{conjecture}
Take $d\in \{1,2,3\}$. The point process $\left(\frac{\lambda_{n,L} + a_L}{b_L}\right)_{n\ge 1}$ converges in law as $L\to\infty$ to a Poisson point process on $\R$ of intensity $e^x dx$. In particular, the r.v.
$$ -\frac{\lambda_{1,L}+a_L}{b_L}\;$$
converges in law to a Gumbel random variable.
\end{conjecture}

\bigskip

Our second conjecture concerns the asymptotic behaviour of the eigenfunctions near their maxima. We let $U_{n,L} \in [-L/2,L/2]^d$ be the point where $|\varphi_{n,L}|$ reaches its global maximum. Let $Q$ be the unique radial positive solution on $\R^d$ of
$$ -\Delta Q - Q^3 = - Q\;.$$
It is known that - up to translations, dilatations and rescalings - $Q$ is the unique optimiser of the Gagliardo-Nirenberg inequality \eqref{Eq:GN}, see~\cite{Frank} and references therein. One can deduce from~\cite[Sec.5]{Lewin} that $\|Q\|_{L^4}^4 = \frac{2d}{C_d}$.

\begin{conjecture}\label{Conj:2}
Take $d\in \{1,2,3\}$. For any $n\ge 1$, the following convergence holds in probability as $L\to\infty$
\begin{align*}
\bigg(\frac1{a_L^{d/4}} |\varphi_{n,L}|\Big(U_{n,L} + \frac{x}{\sqrt{a_L}}\Big), x\in \R^d\bigg) &\Rightarrow \psi_*\;,\\
\bigg(\frac1{a_L} \xi\Big(U_{n,L} + \frac{x}{\sqrt{a_L}}\Big), x\in \R^d\bigg) &\Rightarrow -\frac{{\psi_*}^2}{\|\psi_*\|_{L^4(\R^d)}^2} \sqrt{\frac{2d}{C_d}}\;,
\end{align*}
with $\psi_*(x) = Q(x) / \|Q\|_{L^2}$.
\end{conjecture}
\noindent In Conjecture \ref{Conj:2} the first convergence holds in a space of distributions; the abuse of notation regarding the scaling on $\xi$ shall be interpreted in the distributional sense, i.e. passing the scaling operations to test functions.


\noindent In dimension $1$, these two conjectures were actually proven by Dumaz and Labb\'e~\cite{DL19} (the convergence to a Gumbel r.v.~was proven earlier by McKean~\cite{McKean}). In that case, we have
$$ a_L \sim \Big(\frac{3}{8} \log L\Big)^{2/3}\;,\quad b_L = \frac1{4\sqrt{a_L}}\;,\quad \psi_* = \frac1{\sqrt 2 \cosh}\;,\quad \frac{{\psi_*}^2}{\|\psi_*\|_{L^4(\R^d)}^2} \sqrt{\frac{2d}{C_d}} = \frac{2}{\cosh^2}\;.$$

\bigskip

The present work is organized in the following way. In Section \ref{sec:proofs}, we collect intermediate ingredients and provide the proofs of Theorem \ref{thm:tail_estimates} and Theorem \ref{thm_asymp}, together with the proofs of the ingredients that do not necessitate regularity structures. In Section \ref{sec:LDP}, we prove some technical results on the Anderson Hamiltonian and present the proof of a large deviation estimate stated in Section \ref{sec:proofs}. The Appendix collects some technical results.

\subsection*{Acknowledgements} C.L.~acknowledges financial support from the project SINGULAR ANR-16-CE40-0020-01. We would like to thank the two anonymous referees for their comments and suggestions.

\section{Proofs of the main theorems} \label{sec:proofs}

We start this section by collecting some simple properties of the operator $\cH_L$: actually, we will consider a more general framework where the spatial domain can be taken to be any given square box $Q\subset \R^d$, and where white noise comes with a prefactor $\beta > 0$, as it will be required later on. The second subsection presents a large deviation estimate for the main eigenvalue of $\cH_L$ with a small noise ($\beta \downarrow 0$), along with some information on the associated rate function. These properties of the rate function and some results about the associated variational problem are proved in the third subsection. In the fourth subsection, we provide the proofs of the main theorems. Finally, the last subsection gives some heuristic explanations on Conjecture \ref{Conj:2}.\\

From now on, we call \emph{box} any open bounded square box of $\R^d$ of side-length at least one and $d$ is always assumed to lie in $\{1,2,3\}$.

\subsection{Simple properties of the operator}

Let $Q$ be either a box in $\R^d$ or $\R^d$ itself, take $V\in L^2(Q)$, and define the operator
$$\cH(Q,V) := -\Delta + V\;,\quad \mbox{ on }Q\;,$$
endowed with Dirichlet b.c. It is well-known~\cite{Lewin-Spectral, Helffer} that for $d\le 3$, this operator is self-adjoint and bounded below. When $Q$ is bounded, its spectrum is discrete and we denote its eigenvalues in non-decreasing order by $(\lambda_n(Q,V))_{n\ge 1}$. When $Q = \R^d$, we let $\lambda_1(Q,V)$ be the infimum of its spectrum.

The next result constructs the operator as a limit of regularised versions in the case where $V$ is white noise. As recalled in the introduction, in dimension $d=2,3$ one needs to renormalise these regularised versions for the limit to exist.

Fix some even, smooth function $\rho$ integrating to $1$ and supported in the unit ball of $\R^d$. Set $\rho_\eps := \eps^{-d} \rho(\cdot/\eps)$ for any $\eps > 0$ and define $\xi_\eps := \xi*\rho_\eps$ for some white noise $\xi$. For any $\beta > 0$, we consider the renormalisation constant $C_\eps(\beta)$ associated to the noise $\beta \xi_\eps$: in order not to clutter the presentation, their precise expressions are provided in Appendix \ref{Sec:Renorm}. The proof of the following result is postponed to Subsection \ref{Subsec:Operator} as it requires notions from regularity structures.

\begin{proposition}\label{Prop:Construct}
Fix some parameter $\beta > 0$. There exists a sequence $\eps_k \downarrow 0$ and an event $\Omega_0(\beta)$ of probability one on which the following holds:
\begin{enumerate}
\item For every box $Q$, the sequence $\cH(Q, \beta \xi_{\eps_k} + C_{\eps_k}(\beta) )$ converges in norm resolvent sense to a self-adjoint operator, denoted $\cH(Q, \beta\xi)$ with a slight abuse of notation, with pure point spectrum bounded from below. We denote $(\lambda_n(Q, \beta\xi))_{n \ge 1}$ its eigenvalues in non-decreasing order.
\item For all $n \ge 1$ and all boxes $Q \subset Q'$, we have $\lambda_{n}(Q',\beta\xi) \le \lambda_{n}(Q,\beta\xi)$.
\item For all $L\ge 1$, all $n\ge 1$ and all disjoint boxes $Q_{(1)},\ldots,Q_{(n)} \subset Q$, we have
$$ \lambda_{n}(Q,\beta\xi) \le \max_{1\le i \le n} \lambda_{1}(Q_{(i)},\beta\xi)\;.$$ 
\end{enumerate}
\end{proposition}
\noindent In the particular case $Q=Q_L=(-L/2,L/2)^d$ and $\beta = 1$, we abbreviate $\cH(Q_L, \xi)$ into $\cH_L$ and $\lambda_n(Q_L, \xi)$ into $\lambda_{n, L}$.

\begin{remark}\label{Rk:Op}
Note that we construct the operator $\cH(Q,\xi)$ \emph{simultaneously} for all boxes $Q$, and this allows us to consider almost sure convergence of $\lambda_{n,L}$ as $L\to\infty$. This is to be compared with the construction in~\cite{CvZ19} that holds up to a $\P$-null set that possibly depends on $Q$. The fundamental reason for our simultaneous construction is that all the stochastic objects (the so-called model in the theory of regularity structures) are constructed at once on the full space, while the dependence over the given box $Q$ only goes through the deterministic weights chosen near the boundary of $Q$, see~\cite{Lab19}.
\end{remark}

Let us now collect some simple properties of these operators.

\begin{proposition}[Scaling, independence and invariance properties]\label{prop:scaling}
	\begin{enumerate}
		\item There exists a deterministic constant $\delta_\beta$ such that $\delta_\beta$ tends to $0$ as $\beta \downarrow 0$ and such that for all $L\ge 1$, $\beta > 0$ and $n\ge 1$ the following equality in law holds
		\[ \beta^{2}\lambda_n(Q_L, \xi) =  \lambda_n(Q_{L/\beta}, \beta^{2 - d/2} \xi) + \delta_{\beta} \;.\]
		\item For all disjoint boxes $Q_1,\ldots,Q_k$, the operators $\cH(Q_1,\beta \xi),\ldots,\cH(Q_k,\beta \xi)$ are independent.
		\item For all boxes $Q,Q'$ with the same side-length, $\cH(Q,\beta\xi)$ and $\cH(Q',\beta\xi)$ have the same law.
	\end{enumerate}
\end{proposition}
\begin{remark}
	Properties 2. and 3. should be understood at the level of the resolvents of the operators. Note that the resolvents at stake are random variables taking values in the space of compact operators on $L^2(Q)$, which, equipped with the operator norm, is a separable Banach space.
\end{remark}
\begin{proof}
The second and third properties are consequences of the independence and translation invariance properties of white noise, and of the construction of the operator $\cH(Q,\beta\xi)$ as a limit of regularised operators. We concentrate on the first property. Let $\tilde{\xi}_\eps(x) := \beta^2 \xi_\eps(x\beta)$ for all $x\in\R^d$. Consider the self-adjoint operator
$$ \tilde{\cH}_\eps := -\Delta + \tilde{\xi}_\eps + C_{\eps/\beta}(\beta^{2-\frac{d}{2}})\;,\quad \mbox{ on } Q_{L/\beta}\;.$$
Let $\lambda_{n,L,\eps}$ and $\varphi_{n,L,\eps}$ be the $n$-th eigenvalue and eigenfunction of $-\Delta + \xi_\eps + C_\eps(1)$ on $Q_L$. A computation shows that
$$ \tilde{\cH}_\eps \,\varphi_{n,L,\eps}(\cdot \beta) = \Big(\beta^2 \lambda_{n,L,\eps} + C_{\eps/\beta}(\beta^{2-\frac{d}{2}}) - \beta^2 C_\eps(1)\Big) \varphi_{n,L,\eps}(\cdot \beta)\;.$$
We thus deduce that the $n$-th eigenvalue of $\tilde{\cH}_\eps$ coincides with $\beta^2 \lambda_{n,L,\eps} + C_{\eps/\beta}(\beta^{2-\frac{d}{2}}) - \beta^2 C_\eps(1)$.\\
It turns out that $\tilde{\xi}_\eps$ has the same law as $\beta^{2-\frac{d}{2}} \xi_{\eps/\beta}$, and therefore $\tilde{\cH}_\eps$ has the same law as $\cH(Q_{L/\beta},\beta^{2-\frac{d}{2}} \xi_{\eps/\beta} + C_{\eps/\beta}(\beta^{2-\frac{d}{2}}))$. Passing to the limit along the sequence $\eps_k$ by using Proposition \ref{Prop:Construct}, we deduce the equality in law
$$ \lambda_n(Q_{L/\beta},\beta^{2-\frac{d}{2}}\xi) = \beta^2 \lambda_{n}(Q_L,\xi) - \delta_\beta\;,$$
where $\delta_\beta := \lim_{\eps \downarrow 0} \Big( \beta^2 C_\eps(1) - C_{\eps/\beta}(\beta^{2-\frac{d}{2}})\Big)$. From the asymptotic expressions of the renormalisation constants collected in Appendix \ref{Sec:Renorm}, we can deduce that $\delta_\beta =0$ in dimension $1$ while in dimensions $2$ and $3$, $\delta_\beta = O(\beta^2 \ln \beta^{-1})$ as $\beta\downarrow 0$.
\end{proof}

Finally, we state an estimate that allows to approximate, from above and below, the main eigenvalue over $Q_L$ in terms of the main eigenvalues over smaller boxes. This is a general result, which is originally due to G\"artner and K\"onig~\cite{GK00} in the case where the potential is smooth.

\begin{proposition}[Estimation by division into sub-boxes. See Appendix \ref{sec:boites}]
	There exists a constant $K > 0$ such that for all $\beta > 0$ and all $L > r \ge 1$, we have almost surely
	\begin{equation}
	\min_{k \in \Z^d : |k|_\infty \le \frac{L}{2r} + \frac34} \lambda_1(rk + Q_{3r/2}, \beta\xi) - \frac{K}{r^2} \leq \lambda_1(Q_L, \beta\xi) \leq \min_{k \in \Z^d : |k|_\infty < \frac{L}{2r} - \frac12} \lambda_1(rk + Q_{r}, \beta\xi) 
	\label{eq:boxes}
	\end{equation}
	\label{prop:boites}
\end{proposition}

\subsection{A large deviation estimate}

The proof of the tail estimates stated in Theorem \ref{thm:tail_estimates} revolves around the following large deviations estimate for the main eigenvalue. Its proof requires notions from the theory of regularity structures and is therefore postponed to the next section. From now on, we will often abbreviate $\lambda_1(Q,V)$ into $\lambda(Q,V)$.

\begin{proposition}[See Section \ref{sec:LDP}]
	Fix $L\ge 1$. The collection of random variables $(\lambda(Q_L, \beta\xi))_{\beta > 0}$ satisfies for all $c\in\R$ the following large deviations estimate:
	\begin{align*} -\inf_{x \in (-\infty,c)} I_L(x) &\leq \liminf_{\beta \to 0} \beta^2 \log \P\big(\lambda(Q_L, \beta\xi) \in (-\infty,c)\big)\\
	&\leq \limsup_{\beta \to 0} \beta^2 \log \P\big(\lambda(Q_L, \beta\xi) \in (-\infty,c]\big) \leq -\inf_{x \in (-\infty,c]} I_{L}(x)\;,
	\end{align*}
	where the rate function $I_{L}:\R \to [0, \infty]$ is defined by
	\[I_{L}(x) = \inf \left\{\frac{1}{2} \|V\|^2_{L^2(Q_L)}: V \in L^2(Q_L), \lambda(Q_L, V) = x\right\}\;,\quad x\in\R\;.\]
	\label{prop:PGD}
\end{proposition}
\begin{remark}
For any $x\in \R$ and any box $Q$, there exists $V$ such that $\lambda(Q,V) = x$. Indeed, it suffices to take $V$ constant equal to $x-x_0$ where $x_0 := \lambda(Q,0)$ is the lowest eigenvalue of $-\Delta$ on $Q$.
\end{remark}
We will use the notation $\mathrm{LD}(\beta^2,I_L)$ as a shortcut for the large deviations estimate of rate $\beta^2$ and rate function $I_L$.
We adopt the notation $I_L(J) := \inf_{x\in J} I_L(x)$ for any set $J\subset \R$ and we define the constant
	\begin{equation}
	\label{def:rho} \rho := \inf_{L > 0} I_L ((-\infty, -1]) = \lim_{L\to\infty} I_{L} ((-\infty, -1])\;,
	\end{equation}
where the second equality comes from the fact that $L\mapsto I_L((-\infty,x])$ is non-increasing.

\begin{proposition}[Study of $\rho$. See Section \ref{sec:rho}]
The following properties hold:
	\begin{enumerate}
	\item For all $L\ge 1$, the map $x\mapsto I_L((-\infty,x])$ is continuous.
		\item For all $b \in \R$ and all real sequence $(a_L)_{L \ge 1}$ such that $L^{d/2}a_L \to 0$ as $L \to \infty$, we have
		\[ \lim_{L \to \infty} I_{L}((-\infty, b + a_L]) = \lim_{L \to \infty} I_{L}((-\infty, b ]) = \lim_{L \to \infty} I_{L}((-\infty, b)).\]
		\item The constant $\rho$ can be evaluated through a variational problem :
		\begin{equation}
			\rho = \frac{1}{2} \left\{\sup_{\substack{\psi \in H^1(\R^d) \\ \|\psi\|_{L^2} = 1}} \left(\|\psi\|_{L^4}^2 -  \norm{\nabla \psi}_{L^2}^2\right)\right\}^{-(2-d/2)}
			\label{eq:rho_var}
		\end{equation}
		In particular, we have $C_d = d / \rho = \Big(d^{1+d/2}(4-d)^{2-d/2}/8 \Big)\kappa_d^4$. Moreover, optimizers for this variational problem are properly rescaled optimizers of the Gagliardo-Nirenberg inequality.
	\end{enumerate}
	\label{prop:rho}
\end{proposition}

\subsection{Properties of the rate function} \label{sec:rho}
Recall the Gagliardo-Nirenberg inequality \eqref{Eq:GN} and the associated optimal constant $\kappa_d$ defined in \eqref{Eq:kappa}. In the study of the rate function of the large deviation estimate will appear a variational problem that is closely related to the Gagliardo-Nirenberg inequality: in the following result we collect a few facts on this variational problem.
\begin{lemma}
	We have
	\begin{equation}\label{Eq:Sup}
	\sup_{\substack{\psi \in H^1(\R^d) \\ \norm{\psi}_{L^2} = 1}} \left(\norm{\psi}_{L^4}^2 - \norm{\nabla \psi}_{L^2}^2 \right) = \left(\frac{d}{4}\right)^{\frac{d}{4-d}} \left(\frac{4-d}{4}\right) \kappa_d^{8/(4-d)} = \left(\frac{C_d}{2d}\right)^{\frac1{2-d/2}}\;,
	\end{equation}
	where $C_d$ is given by \eqref{eq:C_d}.
	
	Moreover, a function $w$ with unit $L^2$ norm is an optimizer of \eqref{Eq:Sup} if and only if $w = \lambda^{d/2} u(\lambda \cdot)$ with $\lambda^{2-d/2} = \frac{d \norm{u}_{L^4}^2}{4\norm{\nabla u}_{L^2}^2}$, where $u$ is an optimizer of the Gagliardo-Nirenberg inequality with unit $L^2$ norm.
	
	Finally, we also have
	\begin{align}
	\sup_{\substack{\psi \in H^1(\R^d) \\ \norm{\psi}_{L^2} = 1}} \left(\norm{\psi}_{L^4}^2 - \norm{\nabla \psi}_{L^2}^2 \right) = \sup_{\substack{V \in L^2(\R^d)\\\norm{V}_{L^2} = 1}} -\lambda(\R^d, V)\;.
	\label{eq:rho-2}
	\end{align}
	\label{lem:problem-variationnel}
\end{lemma}
\begin{proof}

For any $u \in H^1(\R^d)$ such that $\|u\|_{L^2}=1$, we write $J(u) := \norm{u}_{L^4}^2 - \norm{\nabla u}_{L^2}^2$. Consider the family of functions $w(x) = \lambda^{d/2} u(\lambda x)$ indexed by $\lambda > 0$. Plugging $w$ into the functional $J$ gives us
\[ J(w) = \lambda^{d/2} \norm{u}_{L^4}^2 - \lambda^2 \norm{\nabla u}_{L^2}^2.\]
for all $\lambda > 0$. The RHS is maximized for $\lambda^{2-d/2} = \frac{d \norm{u}_{L^4}^2}{4\norm{\nabla u}_{L^2}^2}$, which results in the identity
\[ \sup_{\lambda > 0} J(w) = \left(\frac{d}{4}\right)^{\frac{d}{4-d}} \left(\frac{4-d}{4}\right) \left(\frac{\norm{u}_{L^4}}{\norm{\nabla u}_{L^2}^{d/4}}\right)^{\frac{8}{4-d}}.\]
We distinguish two cases. Either $u$ (and therefore all the $w$'s) optimizes the Gagliardo-Nirenberg inequality \eqref{Eq:GN}, in which case
\[ \sup_{\lambda > 0} J(w) = \left(\frac{d}{4}\right)^{\frac{d}{4-d}}\left(\frac{4-d}{4}\right) \kappa_d^{\frac{8}{4-d}}\;.\]
Or $u$ (and therefore all the $w$'s) are not optimizers of the Gagliardo-Nirenberg inequality, and then
\[ \sup_{\lambda > 0} J(w) < \left(\frac{d}{4}\right)^{\frac{d}{4-d}} \left(\frac{4-d}{4}\right) \kappa_d^{\frac{8}{4-d}}\;.\]
This proves \eqref{Eq:Sup}. Moreover, this shows that $w$ is an optimizer of $J$ (over the functions of unit $L^2$ norm) if and only if $w = \lambda^{d/2} u(\lambda \cdot)$ with $\lambda^{2-d/2} = \frac{d \norm{u}_{L^4}^2}{4\norm{\nabla u}_{L^2}^2}$ where $u$ is an optimizer of the Gagliardo-Nirenberg inequality with unit $L^2$ norm.

To prove the final identity of the statement, note that
\begin{align*}
\sup_{\substack{V \in L^2(\R^d)\\\norm{V}_{L^2} = 1}} -\lambda(\R^d, V) = \sup_{\substack{\psi \in H^1\\\norm{\psi}_{L^2} = 1}} \sup_{\substack{V \in L^2(\R^d)\\\norm{V}_{L^2} = 1}} \int \left(-|\nabla \psi|^2 - V\psi^2\right) = \sup_{\substack{\psi \in H^1(\R^d) \\ \norm{\psi}_{L^2} = 1}} \left(\norm{\psi}_{L^4}^2 - \norm{\nabla \psi}_{L^2}^2 \right)
\end{align*}
where we have chosen $V = -\psi^2/\norm{\psi}_{L^4}^2$ to attain the supremum over $V$.
\end{proof}

\begin{proof}[Proof of Proposition \ref{prop:rho}]
1. \& 2. Notice that for $V \in L^2$ and $a \in \R$, we have $\lambda(Q_L, V + a\un_{Q_L}) = \lambda(Q_L, V) + a$ and that for all $\delta > 0$, $\pm 2\langle a\un_{Q_L}, V\rangle = \pm 2\langle (a/\sqrt{\delta})\un_{Q_L}, \sqrt{\delta} V \rangle \leq \delta \norm{V}^2 + \delta^{-1} a^2 L^{d}$. As a result,
\begin{align*}
I_{L} \big((-\infty, b+a]\big) &= \inf\left\{ \frac{1}{2} \norm{V}^2 : \lambda(Q_L, V) \le b + a \right\}\\
&= \inf\left\{ \frac{1}{2} \norm{V}^2 : \lambda(Q_L, V - a\un_{Q_L}) \le b \right\}\\
&= \inf\left\{ \frac{1}{2} \norm{V + a \un_{Q_L}}^2 : \lambda(Q_L, V) \le b \right\}\\
&\begin{cases}
\le (1 + \delta)  I_{L}((-\infty, b]) + \frac12 (1 + \delta^{-1}) a^2 L^d\\
\ge (1 - \delta)  I_{L}((-\infty, b]) + \frac12 (1 - \delta^{-1}) a^2 L^d
\end{cases}
\end{align*}
This easily entails 1. and the first equality of 2. Combining the following inequalities
\[ \inf I_{L}((-\infty, b]) \leq \inf I_{L}((-\infty, b)) \leq \inf I_{L}((-\infty, b - L^{-d}]), \]
with the arguments above yields the second equality of 2.

3. Given \eqref{eq:rho-2}, the statement follows if we can establish
\begin{equation}
\rho = \frac{1}{2} \inf_{\substack{V \in L^2(\R^d)\\ \lambda(\R^d, V) \le -1}} \norm{V}_{L^2}^2\;,
\label{eq:rho-0}
\end{equation}
and 
\begin{align}
\inf_{\substack{V \in L^2(\R^d)\\ \lambda(\R^d, V) \le -1}} \norm{V}_{L^2}^2 =  \inf_{\substack{V \in L^2(\R^d)\\ \lambda(\R^d, V) = -1}} \norm{V}_{L^2}^2 = \left\{\sup_{\substack{W \in L^2(\R^d)\\\norm{W}_{L^2} = 1}} -\lambda(\R^d, W)\right\}^{-(2-d/2)}\;.
\label{eq:rho-3}
\end{align}

Both equalities in \eqref{eq:rho-3} are proven by an argument of scaling. Let us start with the second. By Lemma \ref{lem:problem-variationnel}, the rightmost term of \eqref{eq:rho-3} is strictly positive and therefore
	\begin{equation}
	\left\{\sup_{\substack{W \in L^2(\R^d)\\ \norm{W}_{L^2} = 1}} -\lambda(\R^d, W) \right\}^{-(2-d/2)}
	=	\left\{\sup_{\substack{W \in L^2(\R^d)\\ \norm{W}_{L^2} = 1\\\lambda(\R^d, W) < 0}} -\lambda(\R^d, W) \right\}^{-(2-d/2)}=
	\inf_{\substack{W \in L^2(\R^d)\\ \norm{W}_{L^2} = 1\\ \lambda(\R^d, W) < 0}} \frac{1}{[-\lambda(\R^d, W)]^{2-d/2}}
	\label{eq:rho-4}
	\end{equation}
	Consequently, it suffices to show that
	\begin{equation}\label{Eq:VW} \inf_{\substack{V \in L^2(\R^d)\\ \lambda(\R^d, V) = -1}} \norm{V}_{L^2}^2 = \inf_{\substack{W \in L^2(\R^d)\\ \norm{W}_{L^2} = 1\\ \lambda(\R^d, W) < 0}} \frac{1}{[-\lambda(\R^d, W)]^{2-d/2}}\;.\end{equation}
	Consider the map $V\mapsto W$ defined by: for any $V\in L^2(\R^d)$ such that $\lambda(\R^d, V) = -1$, set $W(x) = r^2 V(rx)$ with $r$ defined by $r^{4-d} = \norm{V}_{L^2}^{-2}$. A simple computation shows that $-\lambda(\R^d,W) = -r^2\lambda(\R^d,V) = r^2$ and $\norm{W}_{L^2}^2 = r^{4-d}\norm{V}_{L^2}^2 = 1$. It is thus immediate to deduce that our map $V\mapsto W$ is a bijection between the two sets that appear in \eqref{Eq:VW}, and by construction, we have $\norm{V}_{L^2}^2 =  \frac{1}{[-\lambda(\R^d, W)]^{2-d/2}}$. The second equality of (\ref{eq:rho-3}) is thus proved.
	
We turn to the first equality in \eqref{eq:rho-3}. Recall that for $d\le 3$, the operator $-\Delta + V$ is bounded below whenever $V\in L^2$. Consequently $\lambda(\R^d,V) > -\infty$. Take $V\in L^2(\R^d)$ such that $\lambda(\R^d, V) < -1$ and set $W(x) = r^2 V(rx)$ with $r$ defined by $r^2 = -1 / \lambda(\R^d,V)$. The previous computations show that $\lambda(\R^d, W) = -1$ and $\norm{V}_{L^2}^2 > \norm{W}_{L^2}^2$. Hence,
	\[ \inf_{\substack{V \in L^2(\R^d)\\ \lambda(\R^d, V) \le -1}} \norm{V}_{L^2}^2 \ge  \inf_{\substack{W \in L^2(\R^d)\\ \lambda(\R^d, W) = -1}} \norm{W}_{L^2}^2. \]
	The converse inequality is obvious.

Now it remains to show (\ref{eq:rho-0}), or in other words,
\[\inf_{\substack{V \in L^2(\R^d)\\ \lambda(\R^d, V) \le -1}} \norm{V}_{L^2}^2 = \inf_{L > 0} \inf_{\substack{V \in L^2(Q_L)\\ \lambda(Q_L, V) \le -1}} \norm{V}_{L^2(Q_L)}^2.\]

The fact that l.h.s. $\leq$ r.h.s. is a direct consequence of the inequality $\lambda(\R^d, V \un_{Q_L}) \le \lambda(Q_L, V)$. On the other hand, fix any $\eps > 0$ and pick $V_\eps \in L^2(\R^d)$ such that $\lambda(\R^d,V_\eps) \le -1$ and
\[\| V_\eps \|_{L^2}^2 < (1 + \eps)  \inf_{\substack{V \in L^2(\R^d)\\\lambda(\R^d, V) \le -1}} \norm{V}_{L^2}^2\;.\]
Let $\bar{V}_\eps := \alpha^2 V_\eps(\alpha \cdot)$ with $\alpha = \sqrt{1 + \eps}$. By the scaling property $\lambda(\R^d, \bar{V}_\eps) = \alpha^2 \lambda(\R^d, V_\eps)$, we have
\begin{align*}
&\norm{\bar{V}_\eps}_{L^2}^2 = \alpha^{4 - d} \norm{V_\eps}_{L^2}^2 < (1+\eps)^{3 - d/2} \inf_{\substack{V \in L^2(\R^d)\\\lambda(\R^d, V) \le -1}} \norm{V}_{L^2}^2 \\
&\lambda(\R^d, \bar{V}_\eps) = \alpha^{2} \lambda(\R^d, V_\eps) \leq -\alpha^{2} = - 1 - \eps
\end{align*}
It can be checked that as $L\to\infty$, $ \lambda(Q_L,\bar{V}_\eps) \to \lambda(\R^d,\bar{V}_\eps)$, and therefore for all $L$ large enough we have
\[\lambda(Q_L,\bar{V}_\eps) \le -1.\]
thus implying for all $L$ large enough
\[\inf_{\substack{V \in L^2(Q_L)\\ \lambda(Q_L, V) \le -1}} \norm{V}_{L^2}^2 \leq \norm{\bar{V}_\eps}_{L^2(Q_L)}^2 \leq \norm{\bar{V}_\eps}_{L^2(\R^d)}^2 \leq (1+\eps)^{3 - d/2}\inf_{\substack{V \in L^2(\R^d)\\ \lambda(\R^d, V) \le -1}} \norm{V}_{L^2}^2.\]
\eqref{eq:rho-0} is then proved by taking an infimum over $L$ and making $\eps$ shrink to zero.

Combining (\ref{eq:rho-0}), (\ref{eq:rho-2}) and (\ref{eq:rho-3}), (\ref{eq:rho_var}) is then established, whereas the expression for $C_d$ and the properties satisfied by the optimizer are consequences of Lemma \ref{lem:problem-variationnel}.
\end{proof}

\subsection{Proofs of Theorems \ref{thm_asymp} and \ref{thm:tail_estimates}}\label{sec:main-proofs}

We can now proceed to the proof of the tail estimates.
\begin{proof}[Proof of Theorem \ref{thm:tail_estimates}]
	We will only consider $n = 1$ as this is the only case needed for the proof of Theorem \ref{thm_asymp}. To treat the general case $n\ge 1$, it suffices to argue as in the proof of Theorem \ref{thm_asymp} below, see in particular \eqref{Eq:1ni}. For $x > 0$, we write $\beta = 1/\sqrt{x}$. We begin by applying the scaling property of Proposition \ref{prop:scaling}: 
	\[\P\left(\lambda_1(Q_L, \xi) \geq -x\right) = \P\left(\beta^2 \lambda_1(Q_L, \xi) \geq -1\right) = \P\left(\lambda_1(Q_{L/\beta}, \beta^{2 - d/2} \xi) + \delta_{\beta} \geq -1\right).\] 
	At this point, for any $1\le r \le L/\beta$ we squeeze the domain $Q_{L/\beta}$ in between unions of boxes of size $r$ and $3r/2$ respectively:
	\[ \bigcup_{k\in\Z^d: |k|_\infty \le \frac{L}{2\beta r}-\frac12} \{rk + Q_r\} \quad\subset \quad Q_{L/\beta} \quad \subset\quad \bigcup_{k\in\Z^d: |k|_\infty \le \frac{L}{2\beta r}+\frac34} \{rk + Q_{3r/2}\}\;.\]
	Note that the number of boxes in both unions is of order $(\frac{L}{\beta r})^d$ as $\beta \downarrow 0$ uniformly over all $L\ge 1$, so that for any given $c_1 < 1 < c_2$ there exists $\beta_1 = \beta_1(r)$ such that for all $\beta < \beta_1$ and all $L\ge 1$ these two numbers are comprised in between $c_1 (\frac{L}{\beta r})^d$ and $c_2 (\frac{L}{\beta r})^d$.
	Applying the estimation by division into small boxes of Proposition \ref{prop:boites} at the first line and the independence property of Proposition \ref{prop:scaling} at the second line, we thus get for all $x \ge x_1 := \beta_1^{-2}$
	\begin{align*}
	\P(\lambda_1(Q_L, \xi) \geq -x) &\leq \P\left(\min_{k \in \Z^d : |k|_\infty < \frac{L}{2 \beta r} - \frac12} \lambda_1(rk + Q_{r}, \beta^{2 - d/2} \xi) \geq -1 -  \delta_{\beta}\right) \\
	&\leq \left[1 - \P\left( \lambda_1(Q_{r}, \beta^{2 - d/2} \xi)  < -1 - \delta_\beta\right)\right]^{c_1 \left(\frac{L}{\beta r}\right)^d}\;.
	\end{align*}
	Regarding the lower bound we also apply Proposition \ref{prop:boites}. However the r.v.~that appear are no longer independent so we use a union bound at the third line to get
	\begin{align*}
	\P(\lambda_1(Q_L, \xi) \geq -x) &\geq \P\left(\min_{k \in \Z^d : |k|_\infty < \frac{L}{2 \beta r} + \frac34} \lambda_1(rk + Q_{3r/2}, \beta^{2 - d/2} \xi) \geq -1 -  \delta_{\beta} + \frac{K}{r^2}\right) \\
	&\geq 1 - \P\left(\min_{k \in \Z^d : |k|_\infty < \frac{L}{2 \beta r} + \frac34} \lambda_1(rk + Q_{3r/2}, \beta^{2 - d/2} \xi) < -1 -  \delta_{\beta} + \frac{K}{r^2}\right) \\
	&\geq 1 - c_2 \left(\frac{L}{ \beta r}\right)^d \P\left( \lambda_1(Q_{3r/2}, \beta^{2 - d/2} \xi)  < -1 - \delta_{\beta} + \frac{K}{r^2}\right)\;.
	\end{align*}
	Since $\delta_{\beta}$ goes to $0$ deterministically as $\beta \to 0$ and given the continuity of the rate function stated as item 1. of Proposition \ref{prop:rho}, the term $\delta_\beta$ does not affect large deviation behaviors. Consequently the large deviations estimate of Proposition \ref{prop:PGD} implies that for fixed $r$ and as $\beta\downarrow 0$
	\[\lambda_1(Q_{r}, \beta^{2 - d/2} \xi) + \delta_{\beta} \sim \mathrm{LD}(\beta^{4 - d}, I_{r}), ~~ \lambda_1(Q_{3r/2}, \beta^{2 - d/2} \xi) + \delta_{\beta} \sim \mathrm{LD}(\beta^{4 - d}, I_{3r/2}).\]
	Moreover the properties on the rate function collected in Proposition \ref{prop:rho} show that
	\[ \rho = \lim_{r \to \infty} I_{r}\big((-\infty, -1)\big) = \lim_{r \to \infty} I_{3r/2}\big((-\infty, -1 + K/r^2]\big) > 0\;.\]
	Fix some $\eta > 0$. Choosing $r$ large enough we thus have
	\[ \rho (1 - \eta) < I_{3r/2}\big((-\infty, -1 + K/r^2]\big)\;,\quad  I_{r}\big((-\infty, -1)\big) < \rho (1 + \eta)\;. \]
	The aforementioned large deviations estimates imply that for all $r$ large enough, we have
	\begin{align*}
	\limsup_{\beta\downarrow 0} \beta^{4-d} \log \P\big(\lambda_1(Q_{3r/2}, \beta^{2 - d/2} \xi)  < -1 - \delta_{\beta} + K/r^2\big) &\leq - I_{3r/2}((-\infty, -1 + K/r^2]) < - \rho(1 - \eta)\;,\\
	\liminf_{\beta\downarrow 0} \beta^{4-d} \log \P\big(\lambda_1(Q_{r}, \beta^{2 - d/2} \xi)  < -1- \delta_{\beta} \big) &\geq - I_{r}((-\infty, -1)) > -\rho(1 + \eta)\;.
	\end{align*}
	We can thus find $\beta_0 \le \beta_1$ such that for all $\beta \le \beta_0$ (that is, for all $x\ge x_0 := \beta_0^{-2}$):
	\[1 - c_2 \left(\frac{L}{\beta r}\right)^d e^{-(1-\eta) \rho \beta^{-4+d}} \leq \P\big(\lambda_1(Q_L, \xi) \geq -x\big) \leq \left(1 - e^{-(1 + \eta)\rho \beta^{-4 + d}}\right)^{c_1 (\frac{L}{\beta r})^d}. \]
	Since we have $e^{-2y} \le 1 - y \le e^{-y}$ for all $y\ge 0$ small enough, the previous inequalities yield (up to possibly diminishing $\beta_0$):
	\[\exp\left(-\frac{2c_2}{\beta^d r^d}  e^{d\log L -(1 - \eta)\rho \beta^{-4 + d}}\right) \leq \P(\lambda_1(Q_L, \xi) \geq -x) \leq \exp\left(-\frac{c_1}{\beta^d r^d}  e^{d\log L - (1 + \eta)\rho \beta^{-4 + d}}\right).\]
	Setting $\gamma_1 = c_1/r^d$, $\gamma_2 = 2c_2/r^d$, and replacing $\beta$ by $1/\sqrt x$, we obtain the desired bound for $n=1$.
\end{proof}

With Theorem \ref{thm:tail_estimates} at hand, we are now able to prove our main result.
\begin{proof}[Proof of Theorem \ref{thm_asymp}]
	In this proof, we set for convenience the quantity $a_L := (C_d \log L)^{\frac1{2-d/2}}$ for $L\ge 1$ (which is the first term in \eqref{eq:a_L}). Assume that for all $\delta > 0$ and all $n\ge 1$
	\begin{equation}
	\P\left( \liminf_{m\to\infty} \left\{ -(1+\delta)\,a_{2^m} \le \lambda_{n, 2^m} \le - (1-\delta)\,a_{2^m} \right\} \right) = 1\;.
	\label{eq:obj}
	\end{equation}
	To deduce the statement of the theorem, it suffices to extend this asymptotic from $L\in \{2^m,m\ge 1\}$ to general $L\in [1,\infty)$. This can be done as follows. For all $L \ge 1$, there exists $m \in \N$ such that $2^m \leq L < 2^{m + 1}$. By item 2. of Proposition \ref{Prop:Construct}, on the event $\Omega_0(1)$ and provided $\lambda_{n,2^m} \le 0$ we have
	\[ \frac{\lambda_{n, 2^{m+1}}}{a_{2^{m}}} \leq \frac{\lambda_{n, L}}{a_L} \leq \frac{\lambda_{n, 2^{m}}}{a_{2^{m+1}}} \;.\]
	Since $a_{2^m} / a_{2^{m+1}} \to 1$ as $m\to\infty$, we deduce from \eqref{eq:obj} that the middle term goes to $-1$ almost surely as $L\to\infty$.\\
	We are left with the proof of \eqref{eq:obj}. By item 3. of Proposition \ref{Prop:Construct} we have on the event $\Omega_0(1)$ and for all $n\ge 1$
	\begin{equation}\label{Eq:1ni}
	\lambda_{1, L} \leq \lambda_{n, L} \leq \max_{1 \leq i \leq n} \lambda_{(i)}
	\end{equation}
	where $\lambda_{(i)}$ is the principal eigenvalue of the operator $\cH(Q_{(i)},\xi)$ and the boxes $Q_{(i)}$ are $n$ disjoint sub-boxes of $Q_L$ whose side-lengths are $L/n$. By item 3. of Proposition \ref{prop:scaling}, the $\lambda_{(i)}$'s are i.i.d.~with the same law as $\lambda_{1, L/n}$. Specialising these inequalities to $L=2^m$, we deduce that \eqref{eq:obj} follows from
	\begin{align}
	&\P\left( \limsup_{m\to\infty} \left\{ \lambda_{1, 2^m} < -(1+\delta)\,a_{2^m} \right\} \right) = 0\;,\label{Eq:limsup1}\\
	&\P\left( \limsup_{m\to\infty} \left\{  \max_{1 \leq i \leq n} \lambda_{(i)} > - (1-\delta)\,a_{2^m} \right\} \right) = 0\;.\label{Eq:limsup2}
	\end{align}
	Take $\eta > 0$ such that $(1-\eta)(1+\delta)^{2-\frac{d}{2}} > 1$. By Theorem \ref{thm:tail_estimates} and using the inequality $1-e^{-y} \le y$ that holds for all $y\in\R$, we find for all $m\ge 1$ large enough
	\begin{align*}
	\P\left(\lambda_{1, 2^m} < -(1+\delta)\,a_{2^m}\right) &= 1-\P\left(\lambda_{1,2^m} \ge -(1+\delta)\,a_{2^m} \right)\\
	&\le \gamma_2 (1+\delta)^{d/2} \left(C_d \log 2^m\right)^{\frac{d/2}{2- d/2}} 2^{-md[ (1-\eta)(1+\delta)^{2-\frac{d}{2}}  - 1]}\;.
	\end{align*}
	The Borel-Cantelli Lemma allows to deduce \eqref{Eq:limsup1}.\\
	Recall that the $\lambda_{(i)}$'s are i.i.d.~with the same law as $\lambda_{1,2^m/n}$ so that
	$$\P\left(\max_{1 \leq i \leq n} \lambda_{(i)} > -(1-\delta)\,a_{2^m}\right) \le n \P\left(\lambda_{1, 2^m / n} > -(1-\delta)\,a_{2^m}\right)\;.$$
	Take $\eta > 0$ such that $(1+\eta)(1-\delta)^{2-\frac{d}{2}} < 1$. By Theorem \ref{thm:tail_estimates} again, we thus find for all $m\ge 1$ large enough
	\begin{align*}
	\P\left(\lambda_{1, 2^m / n} > -(1-\delta)\,a_{2^m} \right) \le \exp\left[ - \gamma_1 n^{-d} (1-\delta)^{d/2} \left(C_d \log 2^m\right)^{\frac{d/2}{2- d/2}} 2^{md[1- (1+\eta)(1-\delta)^{2-\frac{d}{2}}]} \right]\;.
	\end{align*}
	Applying again the Borel-Cantelli Lemma, we deduce \eqref{Eq:limsup2}, thus concluding the proof of Theorem \ref{thm_asymp}.
\end{proof}

\subsection{About Conjecture \ref{Conj:2}}
Recall that the optimizers of the Gagliardo-Nirenberg inequality are exactly the functions $aQ(bx+c)$ with $a,b\in \R\backslash\{0\}, c\in\R^d$ where $Q$ is the unique positive radial solution of $-\Delta Q - Q^3 = -Q$.
\begin{proposition}\label{prop:optimizers}
	The optimizers of \eqref{Eq:Sup} coincide with the set of functions $\{\pm w_*( \cdot + c), c \in \R^d\}$, where
	\[w_* := \frac{\mu^{d/2}}{\|Q\|_{L^2}} Q(\mu\,\cdot), ~~ \mu = \Big(\frac{C_d}{2d}\Big)^{\frac1{4-d}} \;.\]
\end{proposition}
\begin{proof}
	From Lemma \ref{lem:problem-variationnel}, we know that any optimizer $w$ is of the form $w = \lambda^{d/2} u(\lambda \cdot)$ where $\lambda^{2-d/2} = \frac{d \norm{u}_{L^4}^2}{4\norm{\nabla u}_{L^2}^2}$ and $u$ is a Gagliardo-Nirenberge optimizer with unit $L^2$-norm. Also, from the property recalled at the beginning of this subsection, any such $u$ takes the form $aQ(bx + c)$ with $a = \pm |b|^{d/2} \|Q\|_{L^2}^{-1}$, $b \in \R\backslash \{0\}$ and $c\in\R^d$. A computation then shows that $w$ is an optimizer to \eqref{Eq:Sup} if and only if it is of the form $\pm\frac{|b\lambda|^{d/2}}{\norm{Q}_{L^2}}  Q(b\lambda \cdot + c)$ where $|b\lambda|$ is a fixed value given by $\left(\frac{d \norm{Q}_{L^4}^2}{4\norm{\nabla Q}_{L^2}^2}\right)^{\frac{2}{2-d/2}}$. Set $w_* = \frac{|b\lambda|^{d/2}}{\norm{Q}_{L^2}} Q(|b\lambda| \cdot)$. Then indeed all optimizers to \eqref{Eq:Sup} are given by $\{\pm w_*( \cdot + c), c \in \R^d\}$.
	
	The value of $|b\lambda|$ can be determined by an indirect argument. Recall that \eqref{Eq:Sup} can be reformulated into the constrained variational problem $\tilde{J}(w, t) = J(w) - t (\norm{w}_{L^2}^2 - 1)$, where $J$ is the functional introduced in the proof of Lemma \ref{lem:problem-variationnel} and $t \in \R$ is the Lagrange multiplier. Routine variational calculus arguments (taking $\delta \tilde{J} = 0$, see for example \cite[page 9]{Amb92}) show that any optimizer $w$ satisfies the PDE
	\begin{equation*}
	\begin{cases}
	-\Delta w - \frac{w^3}{\|w\|_{L^4}^2} = - \Big(\frac{C_d}{2d}\Big)^{\frac1{2-d/2}}w\\
	\norm{w}_{L^2} = 1
	\end{cases}
	\end{equation*}
	of which $w_*$ is a positive, radial solution. Set $v(x) = (\mu \|w_*\|_{L^4})^{-1} w(x/\mu)$ with $\mu$ as in the statement. Recall from the introduction that $\norm{Q}_{L^4}^4 = \frac{2d}{C_d}$. A computation shows that $v$ is a positive, radial solution of
		\[\begin{cases}
		-\Delta {v} - {v}^3 = - {v}\\
		\norm{{v}}_{L^4}^4 = \mu^{d-4} = \frac{2d}{C_d}\;,
	\end{cases}\]
	and therefore $v= Q$. This leads to the identity
	\[ v(x)=\Big(\frac{|b\lambda|}{\mu}\Big)^{d/4} Q(|b\lambda| x) = Q(x)\;,\]
	which evaluated at $x=0$, and given $Q(0) \ne 0$ since $Q$ is positive, ensures that $|b\lambda| = \mu$ and therefore
	\[ w_*(x) = \frac{|b\lambda|^{d/2}}{\norm{Q}_{L^2}} Q(|b\lambda| x) = \frac{\mu^{d/2}}{\|Q\|_{L^2}} Q(\mu x)\;.\]
%
%
%
%
%
\end{proof}

Take $\beta^2 = 1/a_L$. By the rescaling property stated in Proposition \ref{prop:scaling}, the event $\{ \lambda(Q_L,\xi) \asymp -\beta^{-2}\}$ is ``equivalent'' to the event $\{ \lambda(Q_{L/\beta},\beta^2 \xi(\cdot \beta)) \asymp -1\}$. (Recall that $\beta^2 \xi(\cdot \beta)$ has the same law as $\beta^{2-\frac{d}{2}} \xi$.) By the subdivision into sub-boxes of Proposition \ref{prop:boites}, the latter event ``coincides'' with
$$\Big\{ \min_{|k| \le \frac{L}{2\beta r} } \lambda\big(kr + Q_{r},\beta^2 \xi(\cdot \beta)\big) \asymp -1 \Big\}\;.$$
The r.v.~involved in the $\min$ are independent. For such an event to be satisfied, typically only one of these r.v.~is of order $-1$. The large deviation estimate of Proposition \ref{prop:PGD} shows that the probability of $\{ \lambda(Q_{r},\beta^2 \xi(\cdot \beta)) \asymp -1\}$ is roughly
$$ \exp(-\beta^{d-4} I_r(-\infty,-1))\;.$$
Heuristically, to achieve $\{ \lambda(Q_{r},\beta^2 \xi(\cdot \beta)) \asymp -1\}$, one needs $\beta^2 \xi(\cdot \beta) \asymp V_*$ where $V_*$ is the argmin of $I_r$. Taking $r$ large enough, this argmin should be close to the argmin of $I_\infty$ and we can thus assume that $V_*$ is the argmin of $I_\infty$.\\
The computations in the proofs of Proposition \ref{prop:rho} and Lemma \ref{lem:problem-variationnel} show that $V_*(y) = \mu^{-2} W_*(y/\mu)$, with $\mu =\Big(\frac{C_d}{2d}\Big)^{\frac1{4-d}}$ and $W_*(u) = -\frac{w^2_*(u)}{\|w_*\|_{L^4}^2}$ where $w_*$ is the radial optimizer of \eqref{Eq:Sup} with unit $L^2$ norm defined in Proposition \ref{prop:optimizers}. From Proposition \ref{prop:optimizers}, we deduce $W_* = - \mu^2 Q^2(\mu \cdot)$ and that $V_* = -Q^2 = -\frac{{\psi_*}^2}{\|\psi_*\|_{L^4(\R^d)}^2} \sqrt{\frac{2d}{C_d}}$ where $\psi_* = Q / \|Q\|_{L^2}$. Note that $\lambda(\R^d,V_*) = -1$ and that
\[ -\Delta \psi_* +V_* \psi_* = -\psi_*\;,\]
so that $\psi_*$ is the eigenfunction associated to the smallest eigenvalue of $-\Delta + V_*$.\\
The above discussion suggests that $\xi$ has a deterministic behavior at space-scale $1/\sqrt{a_L}$ around $U_{n,L}$:
\[ \bigg(\frac1{a_L} \xi\Big(U_{n,L} + \frac{x}{\sqrt{a_L}}\Big), x\in \R^d\bigg) \Rightarrow V_*\;.\]
If we assume that for some function $\psi$ (of unit $L^2$ norm) we have
\begin{align*}
\bigg(\frac1{a_L^{d/4}} |\varphi_{n,L}|\Big(U_{n,L} + \frac{x}{\sqrt{a_L}}\Big), x\in \R^d\bigg) \Rightarrow \psi\;,
\end{align*}
then a formal passage to the limit as $L\to\infty$ on
$$ -\Delta \varphi_{n,L} + \xi \varphi_{n,L} = \lambda_{n,L} \varphi_{n,L}\;,$$
yields
$$ - \Delta \psi + V_* \psi = -\psi\;,$$
and we deduce that $\psi=\psi_*$.

\section{Regularity structures and the large deviations estimate}\label{sec:LDP}
Let us briefly summarize the construction of $\cH = \cH(Q,\beta\xi)$ carried out in \cite{Lab19}. It consists in constructing the resolvents $(\cH - a)^{-1}$ associated to this operator. The advantage of dealing with the resolvents is that they satisfy, at least formally, an SPDE that one can hope to solve. Indeed, formal computations show that for any $g \in L^2(Q)$, $(\cH + a)^{-1}g$ should be a fixed point of the map
\begin{equation}\label{Eq:FixedPt} u \mapsto (-\Delta + a)^{-1} g - (-\Delta + a)^{-1} (\beta \xi u)\;,\end{equation}
where $(-\Delta + a)^{-1}$ is the resolvent of the Laplacian endowed with Dirichlet b.c.~on $Q$.
While the above procedure can be performed by standard arguments when $\xi$ is smooth, the case of white noise is singular. Indeed, an iteration of the fixed point map yields a term with the same regularity as $\xi \cdot (-\Delta + a)^{-1} \xi$, and it happens that such a quantity blows up in dimension $d\ge 2$.\\

This is where the theory of regularity structures~\cite{Hai14} is applied. The first ingredient is the notion of \emph{model}: this is a random object that encapsulates the values of the noise but also of non-linear functionals associated to it. For smooth driving noise $\xi$, there is an associated \emph{canonical model} that can be constructed generically. On the other hand, for singular $\xi$ such as white noise, some non-linear functionals are ill-defined and the construction of the model is performed through a limiting procedure: starting from a regularised noise $\xi_\eps$ for which all non-linear functionals are well-defined, one subtracts some renormalisation constants that typically diverge as $\eps\downarrow 0$, but that allow to pass to the limit on the non-linear functionals.\\
The second ingredient is a calculus developed at the level of the so-called modelled distributions, that allows to prove existence and uniqueness of fixed point maps such as \eqref{Eq:FixedPt}: this part of the theory will not be needed in the present work.

\subsection{Regularity structures}\label{Subsec:RS}

\begin{definition}
A regularity structure is a triplet $(\cA, \cT, \cG)$ with the following properties:
\begin{enumerate}
	\item $\cA \subset \R$ is a locally finite set of indices that is bounded from below and contains $0$.
	\item $\cT = \bigoplus_{\alpha \in \cA} \cT_\alpha$ is a graded vector space, where for each $\alpha \in \cA$, $\cT_\alpha$ is a finite-dimensional Banach space equipped with the norm $\norm{\cdot}_{\alpha}$. In particular, we demand $\cT_0 \simeq \R$ with the unit vector $\un$. For any vector $\tau$ in a finite-dimensional subspace of $\cT$, we write $\norm{\tau}$ to be its Euclidean norm. 
	\item $\cG$ is a group acting on $\cT$ such that every element $\Gamma$ of $\cG$ satisfies $\Gamma \un = \un$ and for all $\tau \in \cT_{\alpha}$, $\Gamma\tau - \tau \in \cT_{<\alpha} := \bigoplus_{\beta \in \cA_{<\alpha}} \cT_{\beta}$ where $\cA_{<\alpha} = \cA \cap (-\infty, \alpha)$.
\end{enumerate}
\label{def:reg_struct}
\end{definition}

The prototype for Definition \ref{def:reg_struct} is the polynomial regularity structure, where $\cA = \N$, $\cT_n$ is the vector space spanned by all $X^k$ such that $|k|= n$ and $\cG$ is the group formed by all the transformations $\Gamma_h$ that translate any polynomial $P$ by the vector $h$: $\Gamma_h P(X) = P(X+h)$. (Note that we have used the multi-index notation $X^k = X_1^{k_1} \dots X_d^{k_d}$ and $|k| = k_1 + \dots + k_d$.)\\

Let us now introduce the regularity structure associated to the Anderson Hamiltonian, which is an ``enlargement'' of the polynomial regularity structure. We consider the abstract symbol $\Xi$ that represents $\xi$ at the level of the model space $\cT$. We define the sets $\cF$ and $\cU$, as the smallest sets of symbols such that $\cF$ contains the noise symbol $\Xi$, $\cU$ contains all polynomials $X^k$ and
\[\tau \in \cU \implies \tau \Xi \in \cF ~~\text{and}~~ \tau \in \cF \implies \cI(\tau) \in \cU.\]
Here $\tau\Xi$ and $\cI(\tau)$ denote new symbols that need to be considered. In some sense, $\cU$ allows to describe the solution $u$ of the fixed point problem while $\cF$ allows to describe the product $u\xi$.

Each expression $\tau$ is assigned a number $|\tau|$ called homogeneity, which is calculated by the following rules: (1) $|X^k| = |k|$, (2) $|\Xi| = -d/2 - \kappa$ for some fixed $\kappa \in (0, 1/8)$, (3) for any $\tau, \tau'$, $|\tau\tau'| = |\tau| + |\tau'|$, (4) for any $\tau$, $|\cI(\tau)| = |\tau| + 2$. Note that $\cI$ stands for an integration map associated to our convolution kernel (the Green function of $-\Delta + m$ for some $m>0$, see below), which improves regularity by $2$ and this explains (4).\\
Let therefore $\cA := \{|\tau|: \tau \in \cF \cup \cU\}$ and for $\alpha \in \cA$, and let $\cT_{\alpha}$ be the vector space spanned by all $\tau \in \cF \cup \cU$ such that $|\tau| = \alpha$. For the construction of the structure group $\cG$, we refer the reader to \cite[Sec. 8.1]{Hai14}.\\
From now on, we will \emph{always} restrict $\cU$ (resp.~$\cF$) to symbols whose homogeneities are below some $\gamma \in (3/2, 2-4\kappa)$ (resp.~$\gamma - d/2 - \kappa$). This yields (in the order of increasing homogeneity):
\[\cU =
\begin{cases}
\left\{ \un, X_i, \cI(\Xi) \right\} & \text{ if } d = 1\\
\left\{ \un, \cI(\Xi), X_i \right\} & \text{ if } d = 2\\
\left\{ \un, \cI(\Xi), \cI(\Xi\cI(\Xi)), X_i, \cI(\Xi\cI(\Xi\cI(\Xi))), \cI(\Xi X_i) \right\} & \text{ if } d = 3
\end{cases}\]
On the other hand, the collection $\cF$ is obtained by multiplying the elements in $\cU$ by $\Xi$. The regularity structure $(\cA,\cT,\cG)$ is now fixed once and for all.
\\

We can define the notion of \emph{admissible models} that, roughly speaking, map abstract objects in $\cT$ to genuine distributions. Below we denote by $\ccB^r$ the space of all $\varphi \in C^{r}$ supported in the unit ball of $\R^d$ and such that $\norm{\varphi}_{C^r} \leq 1$, for some $r > 0$. We also introduce the notation $\varphi^{\lambda}_x = \lambda^{-d}\varphi((\cdot - x)/\lambda)$. For every $m>0$, let $P^{(m)}$ be the Green's function of $-\Delta + m$ on $\R^d$, see~\cite[Sec. 3]{Lab19} for its explicit expression. Following~\cite[Def 2.1]{Lab19}, we consider a decomposition of the kernel $P^{(m)}$ into the sum $P^{(m)}_- + P^{(m)}_+$ where $P^{(m)}_-$ is a smooth function that coincides with $P^{(m)}$ outside $B(0,2^{-n_m+1})$ and $P^{(m)}_+$ is compactly supported in $B(0,2^{-n_m+1})$ and coincides with $P^{(m)}$ in $B(0,2^{-n_m-1})$. Here $n_m$ is the smallest integer such that $2^{-n_m} \le 1/\sqrt m$. We refer to Lemma \ref{lem:kernel} for more details.

\begin{definition}
Fix $\gamma \in (3/2,2-4\kappa)$ , $r > -\min \cA = d/2+\kappa$ and $m>0$. An admissible model associated to $P^{(m)}_+$ is a couple $Z = (\Pi, \Gamma)$, where $\Pi = (\Pi_x)_{x \in \R^d}$ is a collection of maps from $\cT$ to the space of distributions $\ccD'(\R^d)$, $\Gamma = (\Gamma_{x, y})_{x, y \in \R^d}$ is a collection of elements of $\cG$, that satisfy the following conditions:
\begin{itemize}
	\item $\Gamma_{x, z} = \Gamma_{x, y} \Gamma_{y, z}$ and $\Pi_{x} = \Gamma_{x,y} \Pi_{y}$ for $x, y, z \in \R^d$.
	\item We have for every $k\in \N^d$, $\Pi_x X^k (y) = (y - x)^k$ together with
	\begin{align*}
	\Pi_x \cI(\tau)(y) = (P^{(m)}_+*\Pi_x \tau)(y) - \sum_{k\in \N^d:|k| < |\cI(\tau)|} \partial^k (P^{(m)}_+*\Pi_x\tau)(x) \frac{(y-x)^k}{k!}\;.
	\end{align*}
	\item For any given compact set $Q \subset\R^d$, we have
	\[ \normm{\Pi^{(m)}}_Q := \sup_{x \in Q} \norm{\Pi^{(m)}}_x < \infty\;,\]
	where
	\[ \norm{\Pi^{(m)}}_{x} := \sup_{\varphi \in \ccB^r} \sup_{\lambda \in (0, 2^{-n_m}]} \sup_{\alpha \in \cA_{<\gamma}} \sup_{\tau \in \cT_\alpha} \frac{\left|\crochet{\Pi^{(m)}_x \tau, \varphi^{\lambda}_x}\right|}{\norm{\tau}_\alpha \lambda^\alpha}\;.\]
	\item For any given compact set $Q \subset\R^d$, we have
	\[ \normm{\Gamma^{(m)}}_Q := \sup_{x, y \in Q : |x - y| \leq 2^{-n_m}} \norm{\Gamma^{(m)}}_{x, y} < \infty\;,\]
	where
	\[ \norm{\Gamma^{(m)}}_{x, y} := \sup_{\alpha \le \beta \in \cA_{<\gamma}}\sup_{\tau \in \cT_\beta}\frac{\norm{\Gamma_{x, y} \tau}_{\alpha}}{\norm{\tau}_\beta |x - y|^{\beta - \alpha}}\;.\]
\end{itemize}
\label{def:model}
\end{definition}

For any box $Q \subset \R^d$, we set
\[\normm{Z^{(m)}}_Q := \normm{\Pi^{(m)}}_Q + \normm{\Gamma^{(m)}}_Q\;.\]
Denote by $\cM_m$ the space of all admissible models with respect to $P^{(m)}_+$. We equip the space $\cM_m$ with the pseudometric
\begin{equation}
\normm{Z;\bar{Z}}_Q := \normm{\Pi - \bar{\Pi}}_Q + \normm{\Gamma - \bar{\Gamma}}_Q
\end{equation}
for $Z = (\Pi, \Gamma)$, $\bar{Z} = (\bar{\Pi}, \bar{\Gamma})$ elements of $\cM_m$ and for any given box $Q$. We can thus consider the quotient space $\cM_{m,Q}$ associated to this pseudometric.\\

For any $V\in L^2_{\tiny\mbox{loc}}(\R^d)$, there exists a unique admissible model $Z^{(m)}(V)=(\Pi^V,\Gamma^V)$ that satisfies $\Pi^V_x \Xi(y) = V(y)$ and for all $\tau \in \cU$, $\Pi^V_x \tau\Xi = (\Pi^V_x \tau)(\Pi^V_x \Xi)$, see Appendix \ref{Appendix:Canonical}. In particular, $V$ can be taken equal to the regularised noise $\beta \xi_\eps := \beta \xi * \rho_\eps$. To alleviate the notations, we will omit the superscript $V$ of $\Pi$ and $\Gamma$, and only express the dependency on the potential function through the notation $Z^{(m)}(V)$ if necessary.\\
Unfortunately, in dimensions $2$ and $3$ the corresponding model does not converge as $\eps\downarrow 0$ to an admissible model. However, it was proved in~\cite{Lab19} that one can build a \emph{renormalised model} $Z^{(m)}_\eps(\beta \xi_\eps)$ associated to $\beta\xi_\eps$ that converges in probability to an admissible model that we denote $Z^{(m)}(\beta\xi)$. This last model can be interpreted as the model associated with the white noise $\beta\xi$. We refer to Appendix \ref{Sec:Renorm} for the definitions of the renormalisation constants and some details on the renormalisation procedure.\\
In fine, for every $m > 0$, for every box $Q$ and every $\delta > 0$, we have
\begin{equation}\label{Eq:CVModel}
\lim_{\eps \to 0} \P\left( \normm{Z^{(m)}_\eps(\beta \xi_\eps); Z^{(m)}(\beta \xi) }_{Q} > \delta \right) = 0.
\end{equation}
For convenience we sometimes only write the shorthands $Z^{(m)}$ and $Z_\eps^{(m)}$ without specifying the potential.

Now fix $\beta \xi$ as potential. For later use, we introduce a sequence $(\eps_k)_{k\ge 1}$ and an event $\Omega'$ of probability one on which we have convergence \emph{simultaneously} on all boxes $Q$ of the renormalised models $Z^{(m)}_{\eps_k}$ towards $Z^{(m)}$. 
It follows from a diagonal argument. Recall that $Q_L = (-L/2, L/2)^d$. We construct the sequence $(\eps_k)_{k \ge 1}$ as follows:
\begin{enumerate}
	\item For $n = 1$, we pick a sequence $(\eps_{1,k})_{k\ge 1}$ such that $\P\left(\normm{Z^{(m)}_{\eps_{1,k}}; Z^{(m)}}_{Q_1} \to 0\right) = 1$.
	\item Having chosen the sequence $(\eps_{n, k})_{k\ge 1}$ such that $\displaystyle\P\left(\bigcap_{i = 1}^n \left\{\normm{Z^{(m)}_{\eps_{i,k}}; Z^{(m)}}_{Q_i} \to 0\right\}\right) = 1$, we choose a subsequence $(\eps_{n+1, k})_{k\ge 1}$ of $(\eps_{n, k})_{k\ge 1}$ such that $\displaystyle\P\left(\bigcap_{i = 1}^{n+1} \left\{\normm{Z^{(m)}_{\eps_{i,k}}; Z^{(m)}}_{Q_i} \to 0 \right\}\right) = 1$.
	\item Take $\eps_k := \eps_{k, k}$.
\end{enumerate}
Let $\Omega_m := \{\forall \mbox{ box }Q: \normm{Z^{(m)}_{\eps_{k}}; Z^{(m)}}_{Q} \to 0 \quad\mbox{ as } k\to\infty\}$. Since $(\eps_k)_{k\ge n}$ is a subsequence of $(\eps_{n,k})_{k\ge n}$ we obtain
\begin{align*}
\P(\Omega_m^\cc) &= \P\left( \exists Q : \normm{Z^{(m)}_{\eps_{k}}; Z^{(m)}}_{Q} \not\to 0 \quad\mbox{ as } k\to\infty \right)\\
&= \P\left(\exists n\ge 1: \normm{Z^{(m)}_{\eps_{k}}; Z^{(m)}}_{Q_n} \not\to 0 \quad\mbox{ as } k\to\infty\right)\\
&\leq \P\left(\exists n\ge 1: \normm{Z^{(m)}_{\eps_{n,k}}; Z^{(m)}}_{Q_n} \not\to 0 \quad\mbox{ as } k\to\infty\right) = 0\;,
\end{align*}
i.e. $\P(\Omega_m) = 1$ for all $m \in \N$. We thus define the event $\Omega' := \cap_{m\ge 1} \Omega_m$ of probability one.

%

\subsection{Construction of the operator}\label{Subsec:Operator}

The operator $\cH(Q_L,\beta\xi)$ will be defined as the limit of the operators $-\Delta + \beta \xi_\eps + C_\eps(\beta)$ as $\eps\downarrow 0$, where $C_\eps(\beta)$ is the renormalisation constant introduced in Appendix \ref{Sec:Renorm}. The limit will be taken in the norm resolvent sense: for all $a$ large enough, the operator $(-\Delta + \beta\xi_\eps + C_\eps(\beta) + a)^{-1}$ converges in norm to some limit that we denote $(-\Delta + \beta\xi + a)^{-1}$. This limit happens to be invertible and self-adjoint, and the operator $-\Delta + \beta\xi$ can then be defined by inversion.\\

Recall that $\cM_{m,Q}$ is the space of all admissible models associated with the convolution kernel $P^{(m)}_+$ and restricted to $Q$. For any constant $K>0$ we define the following subset
$$ \cM_{m,Q,K} := \{ Z\in \cM_{m,Q}: \normm{Z; Z(0)}_{Q} \leq K \}\;,$$
where $Z^{(m)}(0)$ denotes the admissible model associated with the null potential $V \equiv 0$ and the kernel $P^{(m)}_+$. Note that $\cM_{m, Q, K}$ is a closed subset of $\cM_{m, Q}$.

\begin{proposition}\label{Prop:Resolv}
	Fix $m>0$ and a box $Q$. There exists a constant $K=K(m) > 0$ and a continuous map $\Phi_{m,Q}$ from $\cM_{m,Q,K} \times (-2,2)$ into the set of bounded operators on $L^2(Q)$ endowed with the operator norm topology, such that:\begin{enumerate}
	\item for any $V\in L^2(Q)$, if the canonical model $Z(V)$ belongs to $\cM_{m,Q,K}$ then for any $b\in (-2,2)$ we have $\Phi_{m,Q}(Z,b) = (-\Delta + V + m + b)^{-1}$,
	\item if the renormalised model $Z^{(m)}_\eps(\beta \xi_\eps)$ belongs to $\cM_{m,Q,K}$, then for any $b\in (-2,2)$ we have $\Phi_{m,Q}(Z^{(m)}_\eps,b) = (-\Delta + \beta\xi_\eps + C_\eps(\beta) + m + b)^{-1}$.
\end{enumerate}
	Moreover, the constant $K(m)$ increases to infinity with $m$.
\end{proposition}
\begin{proof}
This is a consequence of the fixed point result~\cite[Prop. 3.7]{Lab19}, combined with the reconstruction theorem~\cite[Thm. 3]{Lab19} and some computations as in~\cite[Lem 4.5]{Lab19}.
\end{proof}

\begin{proof}[Proof of Proposition \ref{Prop:Construct}]
Fix $\beta > 0$. Properties 2. and 3. come from standard applications of the min-max formula when the potential is in $L^2(Q)$, in particular for $\beta\xi_\eps + C_\eps(\beta)$. Consequently, on the event where 1. holds, one can pass to the limit on $\eps_k$ and deduce 2. and 3. We are therefore left with the proof of Property 1.

Let us introduce the event
$$ \tilde{\Omega}_{m,n} := \{ Z^{(m)}(\beta \xi) \in \cM_{m,Q_n,\frac12 K(m)} \}\;.$$
It was proven\footnote{The proof is presented for $\beta = 1$ in that reference but the arguments are exactly the same for a general $\beta$.} in~\cite[Prop 4.1]{Lab19} that for any given $n\ge 1$ and any given constant $C>0$, we have $\P(Z^{(m)}(\beta \xi) \in \cM_{m,Q_n,C}) \to 1$ as $m\to \infty$ sufficiently fast so that $\P(\liminf_{m \to \infty} \tilde{\Omega}_{m, n}) = 1$ by Borel-Cantelli. Consequently, the event
$$ \tilde{\Omega} := \bigcap_{n\ge 1} \liminf_{m \to \infty} \tilde{\Omega}_{m,n}\;,$$
has probability one.

We now argue deterministically on the event $\Omega'\cap \tilde{\Omega}$ of probability one. For any box $Q$, taking $n\ge 1$ such that $Q \subset Q_n$, there exists $m_0 \ge 1$ such that for all $m \ge m_0$
$$ Z^{(m)} \in \cM_{m,Q_n,\frac12 K(m)} \subset \cM_{m,Q,\frac12 K(m)}\;,$$
and
$$ \normm{Z^{(m)}_{\eps_{k}}; Z^{(m)}}_{Q}  \to 0\;,\quad k\to\infty\;.$$
Consequently for all $k$ large enough we have
$$ Z^{(m)}_{\eps_k} \in \cM_{m,Q_n,K(m)} \subset \cM_{m,Q,K(m)}\;.$$
Note that for all $a > m_0 - 2$, we can always find a pair $(m, b) \in \N \times (-2, 2)$ such that $a = m + b$ and $m\ge m_0$. For all such $a$, we define
$$ \cG_{a} := \Phi_{m,Q}(Z^{(m)},b)\;,\quad \cG_{a}^{\eps_k} := \Phi_{m,Q}(Z^{(m)}_{\eps_k},b)\;.$$
Proposition \ref{Prop:Resolv} thus ensures that $\cG_{a}^{\eps_k}$ converges in norm to $\cG_{a}$ as $k\to\infty$. Note that this definition of $\cG_{a}$ does not depend on the choice of the pair $(m,b)$ made above: indeed, choosing some alternative parameters $(m',b')$ we have
\begin{align*}
\Phi_{m',Q}(Z^{(m')}_{\eps_k},b') &= (-\Delta + \beta \xi_{\eps_k} +  m' + C_{\eps_k}(\beta) + b')^{-1}\\
&= (-\Delta + \beta \xi_{\eps_k} + m + C_{\eps_k}(\beta)+ b)^{-1} = \Phi_{m,Q}(Z^{(m)}_{\eps_k},b)\;,
\end{align*}
so that by passing to the limit on $k$ we get $\Phi_{m,Q}(Z^{(m)},b) = \Phi_{m',Q}(Z^{(m')},b')$.\\
The arguments in~\cite[Prop 4.2]{Lab19} finally show that there exists a self-adjoint operator $\cH(Q,\beta\xi)$ with pure point spectrum bounded from below whose resolvent is given by $(\cH(Q,\beta\xi)+a)^{-1} = \cG_{a}$. In addition, $\cG_{a}^{\eps_k} = (\cH(Q,\beta\xi_{\eps_k}) + C_{\eps_k}(\beta)) + a)^{-1}$ converges in norm to $(\cH(Q,\beta\xi)+a)^{-1}$ thus concluding the proof.
\end{proof}

\subsection{Proof of Proposition \ref{prop:PGD}}

We now turn to the proof of the large deviations estimate of Proposition \ref{prop:PGD}. From now on, we work on the box $Q_L$ for some $L\ge 1$ fixed once and for all.

The proof relies on the following result of Hairer and Weber. We would like to point out that, although this result is originally spelled out in the regularity structure built for the dynamical $\Phi^4_3$ model, the proof relies on a large deviations principle for Wiener chaos, and, as mentioned on~\cite[p.59]{HW14}, it is applicable to all equations that can be treated with the theory of regularity structures, in particular to our case.

\begin{theorem}[Theorem 4.3 in \cite{HW14}]
	Fix $m \ge 1$. The collection $(Z^{(m)}(\beta\xi))_{\beta \in (0,1]}$ of admissible models in $\cM_{m,Q_L}$ satisfies a large deviation principle with rate $\beta^2$ and rate function
	\[I_{\cM}(Z) := \inf\left\{ \frac{1}{2} \norm{V}_{L^2(Q_L)}^2 : Z^{(m)}(V) = Z \right\}\]
	where $Z^{(m)}(V)$ is the canonical model associated to the deterministic noise $V$.
	\label{lem:PGD_modele}
\end{theorem}

For $m\ge 1$ introduce the event $E_{m,L}(\beta)$ on which the operator $\Phi_{m,Q_L}(Z^{(m)}(\beta \xi),0)$ is well-defined and positive, i.e.~$E_{m, L}(\beta)$ is the event on which
\begin{enumerate}
	\item the model $Z^{(m)}(\beta \xi)$ belongs to $\cM_{m,Q_L,K(m)}$, where $K(m)$ is the constant of Proposition \ref{Prop:Resolv},
	\item $\Phi_{m,Q_L}(Z^{(m)}(\beta \xi),0)$ is a positive operator on $L^2(Q_L)$.
\end{enumerate}
The event $E_{m,L}(\beta)$ occurs with high probability, as shown by the following estimate.
\begin{lemma}\label{Lemma:E}
	There exist some constants $\alpha, M,\nu >0$ such that for all $m \ge 1$ large enough, all $L\ge 1$ and all $\beta \in (0,1]$ we have
	\[ \P\left((E_{m,L}(\beta))^\cc\right) \le ML^d e^{-\frac{\alpha m^\nu}{\beta^2}}\;.\]
	\label{lem:lab_4.10}
\end{lemma}
\begin{proof}[Proof of Lemma \ref{lem:lab_4.10}]
	Suppose that $Z^{(m)} = Z^{(m)}(\beta \xi)\in \cM_{m,L,\frac12 K(m)}$, then the arguments in the proof of Proposition \ref{Prop:Construct} combined with Proposition \ref{Prop:Resolv} show that $(\cH(Q,\beta\xi) + a)$ is invertible for all $a \in (m -2,m+2)$. We thus deduce that if for every $m\ge m_0$, we have $Z^{(m)} \in \cM_{m,L,\frac12 K(m)}$, then $(\cH(Q,\beta\xi) + a)$ is invertible for all $a \in (m_0 -2, +\infty)$, thus implying that all these operators $(\cH(Q,\beta\xi) + a)^{-1}$ are positive, in particular $\Phi_{m,Q_L}(Z^{(m)},0)$.\\
	Now let $m_0 \in \N$ be arbitrary. We deduce that
	\begin{align*}
	\P\left((E_{m_0,L})^\cc\right) &\leq \P\left( \sup_{m \ge m_0} \normm{Z^{(m)} ; Z(0)}_{Q_L} \ge \frac12 K(m) \right) \leq \sum_{m \ge m_0} \P\left( \normm{Z^{(m)}; Z(0)}_{Q_L} \ge \frac12 K(m) \right)\;.
	\end{align*}
	Write $Z^{(m)}(\beta \xi) = (\Pi^{(m), \beta}, \Gamma^{(m), \beta})$. By~\cite[Lemma 2.3 and Lemma 3.1]{Lab19}, there exists a constant $C>0$ such that for all $m\ge 1$ we have
	\[ \P\left( \normm{Z^{(m)}; Z(0)}_{Q_L} \ge \frac12 K(m) \right) \le \P\left( \normm{\Pi^{(m), \beta}}_{\Lambda,Q_L} \ge \frac12 C K(m) \right)\;,\]
	where
	\[ \normm{\Pi^{(m), \beta}}_{\Lambda,Q_L} := \sup_{\zeta \in \cA_{<0}} \sup_{\tau \in \cA_\zeta} \normm{\Pi^{(m), \beta} \tau}_{\Lambda,Q_L}\;,\]
	and, for some scaling function $\varphi$ of a compactly supported wavelet basis
	\[ \normm{\Pi^{(m), \beta} \tau}_{\Lambda,Q_L} := \sup_{n\ge n_m} \sup_{x\in Q_L \cap (2^{-n}\Z^d)} \frac{\left|\crochet{\Pi^{(m), \beta}_x \tau, \varphi^{2^{-n}}_x}\right|}{\norm{\tau}_\zeta 2^{-n\zeta}}\;.\]
	Let $\Pi^{(m)} = \Pi^{(m),1}$ be the model associated to the noise $\xi$ with scaling factor $\beta=1$. We note that $\Pi^{(m),\beta}_x \tau = \beta^{\|\tau\|} \Pi^{(m)}_x \tau$, where $\|\tau\|$ is the number of occurrences of $\Xi$ in the symbol $\tau$. By~\cite[Lemma 4.11]{Lab19} there exist $\lambda,\nu > 0$ such that
	\[ M := \sup_{m \ge 1} \sup_{L \ge 1} \frac{1}{L^d} \sum_{\tau \in \cA_{<0}} \E\left[\exp\left(\lambda m^{\nu} \normm{\Pi^{(m)} \tau}_{\Lambda,Q_L}^{\frac{2}{\norm{\tau}}}\right) \right] < \infty\;.\]
	A simple computation thus yields the existence of a constant $\alpha >0$ such that for all $m\ge 1$ and all $L\ge 1$
	\[ \P\left( \normm{\Pi^{(m), \beta}}_{\Lambda,Q_L} \ge \frac12 C K(m) \right) \le M L^d e^{-\alpha \frac{m^\nu}{\beta^2}}\;,\]
	which concludes the proof upon summing for $m \ge m_0$.
\end{proof}

\begin{proof}[Proof of Proposition \ref{prop:PGD}]
	Fix any $m\ge 1$. On the event $E_{m,L}(\beta) \cap \Omega' \cap \tilde{\Omega}$ (recall the proof of Proposition \ref{Prop:Construct}), the operator $\Phi_{m,Q_L}(Z^{(m)}(\beta),0)$ coincides with the operator $(\cH(Q_L,\beta\xi) + m)^{-1}$ ; it is positive, self-adjoint and compact, implying that its spectrum consists of positive eigenvalues that we denote $(\mu_{n,L})_{n\ge 1}$ in non-increasing order. These eigenvalues can be related to those of $\cH(Q_L,\beta\xi)$, that we denote $(\lambda_{n,L})_{n\ge 1}$ in non-decreasing order, through:
	\[\lambda_{n,L} = (\mu_{n,L})^{-1} - m\;.\]
	Fix some $c\in \R$ and observe that $I_L((-\infty,c])$ and $I_L((-\infty,c))$ are finite.\\
	
	\noindent\emph{Upper bound.} Recall the constant $K=K(m)$ of Proposition \ref{Prop:Resolv}. On the set $\cM_{m,Q_L,K}$, we define the continuous mapping $\varphi_{m} : \cM_{m,Q_L,K} \to \R$ by $Z \mapsto \mu$ where $\mu$ is the supremum of the spectrum of the operator $\Phi_{m,L}(Z,0)$. The continuity of $\varphi_{m}$ is a consequence of the continuity of $\Phi_{m,L}$ and of the variational formula for the supremum of the spectrum. Take $m$ large enough such that $m+c > 0$, we can write\footnote{Implicitly any subset of $\cM_{m,Q_L,K}$ is viewed as a subset of $\cM_{m,Q_L}$: in particular, $\varphi_{m}^{-1}(A)$ for any set $A\subset \R$.}
	\begin{align*}
	\P\big(\lambda_{1,L} \le c ; E_{m,L}(\beta)\big) &= \P\big(\mu_{1,L} \ge \frac1{m+c} ; E_{m,L}(\beta)\big)\\
	&= \P\Big(Z^{(m)}(\beta \xi) \in \varphi_{m}^{-1}\Big(\big[\frac1{m+c},\infty\big)\Big) ; E_{m,L}(\beta)\Big)\;.
	\end{align*}
	Therefore
	\[ \P(\lambda_{1,L} \le c ) \le \P\Big(Z^{(m)}(\beta \xi) \in \varphi_{m}^{-1}\big([\frac1{m+c},\infty)\big)\Big) + \P\left(E_{m,L}(\beta)^\cc\right)\;.\]
	Assume that we can show that for all $m\ge 1$, we have
	\begin{equation}\label{Eq:ZmUpper}
	\limsup_{\beta \to 0} \beta^2 \log\P\Big(Z^{(m)}(\beta \xi) \in \varphi_{m}^{-1}\big([\frac1{m+c},\infty) \big) \Big) \le - I_L((-\infty,c])\;.
	\end{equation}
	Then by Lemma \ref{Lemma:E} we deduce that
	\begin{align*}
	\limsup_{\beta \to 0} \beta^2 \log \P(\lambda_{1,L} \le c) \leq \max\left\{- I_L((-\infty,c]), -cm^\nu \right\}.
	\end{align*}
	Choosing $m$ large enough so that $\alpha m^\nu >  I_{L}((-\infty,c])$, we obtain the desired upper bound of the large deviation estimate. We are left with proving \eqref{Eq:ZmUpper}.\\
	
	Since $\cM_{m,Q_L,K}$ is a closed subset of $\cM_{m,Q_L}$, $\varphi_{m}^{-1}\big([\frac1{m+c},\infty)\big)$ is itself a closed subset of $\cM_{m,Q_L}$. The large deviations principle stated in Theorem \ref{lem:PGD_modele} thus yields
	\[\limsup_{\beta \to 0} \beta^2 \log\P\Big(Z^{(m)}(\beta \xi) \in \varphi_{m}^{-1}\big([\frac1{m+c},\infty)\big) \Big) \le - I_\cM( \varphi_m^{-1}([\frac1{m+c},\infty)))\;.\]	
	At this point, observe that for any $V \in L^2(Q_L)$ such that $Z^{(m)}(V) \in \cM_{m, Q_L, K}$, the operator $(-\Delta + V +m)^{-1}$ is well-defined and $\mu = \varphi_m(Z^{(m)}(V))$ is its largest positive eigenvalue. A priori we do not know whether this operator is positive so that $\mu$ does not necessarily correspond to the smallest eigenvalue $\lambda_1(Q_L,V)$ of $-\Delta +V$ and we only have the inequality
	\[ \lambda_1(Q_L,V) \le \frac{1}{\mu} - m\;.\]
	As a consequence we find
	\begin{align*}
	 I_\cM\left(\varphi_{m}^{-1}([\frac1{m+c},\infty))\right) &= \inf\left\{ \frac{\norm{V}_{L^2(Q_L)}^2}{2} : Z^{(m)}(V) \in \cM_{m, Q_L, K},  \frac{1}{\mu} -m \le c\right\} \\
	&\ge \inf\left\{ \frac{\norm{V}_{L^2(Q_L)}^2}{2} : \lambda_1(Q_L,V) \le c\right\}\;.
	\end{align*}
	The above implies
	\[\limsup_{\beta \to 0} \beta^2 \log \P(\lambda_{1,L} \le c) \leq - I_{L}((-\infty,c])\;,\]
	as required.\\
	
	\noindent\emph{Lower bound.} We consider the \emph{open} subset of $\cM_{m,Q_L}$ defined by
	$$ \cM_{m,Q_L,K_-} := \left\{ Z\in \cM_{m,Q_L}: \normm{Z; Z^{(m)}(0)}_{Q_L} < K \right\}\;.$$
	with the constant $K=K(m)$ of Proposition \ref{Prop:Resolv}. On the set $\cM_{m,Q_L,K_-}$, we consider the continuous mapping $\varphi_{m} : \cM_{m,Q_L,K_-} \to \R$ by $Z \mapsto \mu$ where $\mu$ is the supremum of the spectrum of the operator $\Phi_{m,L}(Z,0)$. Here again $\varphi_m$ is continuous. Again take $m$ large enough such that $m+c > 0$. From similar arguments as before we have
	\[ \P(\lambda_{1,L} < c ) \ge \P\Big(Z^{(m)}(\beta \xi) \in \varphi_{m}^{-1}\big((\frac1{m+c},\infty)\big)\Big) - \P\left(E_{m,L}(\beta)^\cc\right)\;.\]
	Note that $\varphi_{m}^{-1}((\frac1{m+c},\infty))$ is an open set. Therefore the large deviations principle stated in Theorem \ref{lem:PGD_modele} yields
	\begin{equation*}
	\liminf_{\beta \to 0} \beta^2 \log \P\Big(Z^{(m)}(\beta \xi) \in \varphi_{m}^{-1}\big((\frac1{m+c},\infty)\big)\Big) \ge - I_{\cM}\left(\varphi_{m}^{-1}\big((\frac1{m+c},\infty)\big) \right)\;.
	\end{equation*}
	Note that
	\[ \inf I_\cM\left(\varphi_{m}^{-1}\big((\frac1{m+c},\infty)\big)\right) = \inf\left\{ \frac{\norm{V}_{L^2(Q_L)}^2}{2} : Z^{(m)}(V) \in \cM_{m, Q_L, K_-},  \frac{1}{\mu} - m < c\right\} \;.\]
	We would like to compare this quantity to $I_{L}\left((-\infty,c)\right)$: here again the operator $(-\Delta + V + m)^{-1}$ may not be positive so that the supremum of its spectrum may not be related to $\lambda_1(Q_L,V)$.\\
	Recall that $I_L((-\infty,c)) < \infty$ so for any given $\delta > 0$, there exists some function $V \in L^2$ such that $\lambda(Q_L, V) \in (-\infty,c)$ and such that $\norm{V}_{L^2}^2/2 \leq \inf I_{L}((-\infty,c)) + \delta$. Recall from Proposition \ref{Prop:Resolv} that the constant $K(m)$ increases to infinity as $m \to \infty$; hence for this particular $V$, by \eqref{Eq:BdCanonical} we can find $m$ sufficiently large such that $Z^{(m)}(V) \in \cM_{m, Q_L, K}$ and such that $\lambda_1(Q_L, V) + m > 0$. This implies $1/\mu - m = \lambda_1(Q_L, V) \in (-\infty,c)$ and thus proves that for $m$ large enough
	\[ I_\cM\left(\varphi_{m}^{-1}\big((\frac1{m+c},\infty)\big)\right) \le I_{L}((-\infty,c)) + \delta\;.\]
	On the other hand, by Lemma \ref{Lemma:E} we know that, provided $m$ is large enough, $\P\left(E_{m,L}(\beta)^\cc\right)$ is negligible compared to the term that we have just controlled and therefore
	\[ \liminf_\beta \beta^2 \log \P(\lambda < c) \ge - I_{L} ((-\infty,c)) - \delta \;,\]
	for any given $\delta > 0$. This suffices to conclude.
\end{proof}

\appendix

\section{Renormalisation constants}\label{Sec:Renorm}

We let $G$ be the Green's function of $-\Delta$, and $P^{(m)}$ be the Green's function of $-\Delta+m$, we refer to~\cite[Sec 3.1]{Lab19} for the expressions. Recall that $n_m$ is the smallest integer such that $2^{-n_m} \le 1/\sqrt m$.

\begin{lemma}
Fix $r>0$. For every $m\ge 1$, there exists a decomposition $P^{(m)} = P^{(m)}_+ + P^{(m)}_-$ such that:\begin{enumerate}
\item $P^{(m)}_+$ is supported in $B(0,2^{-n_m+1})$ and satisfies $P^{(m)}_+ = P^{(m)}$ on $B(0,2^{-n_m-1})$, while $P^{(m)}_-$ is $\cC^\infty$ and vanishes on $B(0,2^{-n_m-1})$.
\item For all $k\in \N^d$ such that $|k|< r$ we have $\int x^k P^{(m)}_+(x) dx = 0$.
\item There exists a constant $C>0$, independent of $m\ge 1$, such that for all $k\in \N^d$ such that $|k|< r$ we have
\[ |\partial^k P^{(m)}_+(x)| \le C | \partial^k G(x)|\;,\quad x\in\R^d\;.\]
\end{enumerate}
\label{lem:kernel}
\end{lemma}
\begin{proof}
This is a consequence of~\cite[Lemma 3.1]{Lab19} except for the property ``$P^{(m)}_+ = P^{(m)}$ on $B(0,2^{-n_m-1})$'' that was not stated there. However, this property follows if one picks carefully the functions $\eta_k$ in that proof: namely, it suffices to impose to the functions $\eta_k$ to be supported in $B(0,1) \backslash B(0,1/2)$. This can always be achieved, see for instance~\cite[Lemma 8.1]{CarZam}.
\end{proof}

We introduce the renormalisation constants as follows. In dimension $1$, we set $C_\eps^{(m)}(\beta) := 0$. In dimension $2$, we set
$$ C_\eps^{(m)}(\beta) := \beta^2 \int_{\R^2} P_+^{(m)}(x) \rho_\eps^{*2}(x) dx\;.$$
A computation shows that there exists a constant $\tilde{c}_\rho(m)$ independent of $\beta$ such that $C_\eps^{(m)}(\beta) = \beta^2 (2\pi)^{-1} \ln \eps^{-1} + \beta^2 \tilde{c}_\rho(m) + o(1)$ as $\eps\downarrow 0$.\\
In dimension $3$, we set $C_\eps^{(m)}(\beta) := \beta^2 c_\eps^{(m)} + \beta^4 c_\eps^{(m),1,1} + \beta^4 c_\eps^{(m),1,2}$ where
\begin{align*}
c_\epsilon^{(m)} &:= \int P_+^{(m)}(x) \rho_\eps^{*2}(x) \,dx\;,\\
c^{(m),1,1}_\epsilon &:= \iiint P_+^{(m)}(x_1) P_+^{(m)}(x_2) P_+^{(m)}(x_3) \rho_\eps^{*2}(x_1+x_2)\rho_\eps^{*2}(x_2+x_3) \,dx_1 \,dx_2 \,dx_3\;,\\
c^{(m),1,2}_\epsilon &:= \iiint P_+^{(m)}(x_1) P_+^{(m)}(x_2) \big(P_+^{(m)}(x_3) \rho_\eps^{*2}(x_3) - c_\epsilon \delta_0(x_3) \big)\rho_\eps^{*2}(x_1+x_2+x_3) \,dx_1 \,dx_2 \,dx_3\;.
\end{align*}
There exist some constants $c_\rho, \tilde{c}_\rho, \tilde{c}_\rho^{1,1}(m),c_\rho^{1,2}$ independent of $\beta$ such that as $\eps\downarrow 0$
\begin{align*}
c_\epsilon^{(m)} &= \frac{c_\rho}{\eps} + \tilde{c}_\rho \sqrt{m} + o(1)\;,\\
c^{(m),1,1}_\epsilon &= \ln \frac1{\eps} + \tilde{c}_\rho^{1,1}(m) + o(1)\;,\\
c^{(m),1,2}_\epsilon &= c_\rho^{1,2} + o(1)\;.
\end{align*}
Note that the only constant that depends on $m$ is $\tilde{c}_\rho^{1,1}(m)$, and its expression is a bit involved so we refrain from writing it explicitly. On the other hand, if we let $G(x) = \frac1{4\pi |x|}$ (which is nothing but the Green's function of $-\Delta$), we have $c_\rho = \int_{\R^3} G(y) \rho^{*2}(y) dy$, $\tilde{c}_\rho =- \int_{\R^3} G(y)|y| \rho^{*2}(y) dy$, and
\[ c_\rho^{1,2} = \iiint G(y_1) G(y_3)\left( G(z_2 - y_3) - G(z_2) - \crochet{\nabla G(z_2), y_3} \right) \rho^{*2}(y_3) \rho^{*2}(y_1 + z_2) ~dy_1 dz_2 dy_3.\]

The construction of the renormalised model $Z_\eps^{(m)}(\beta)$ follows along the lines of~\cite{HaiPar} and~\cite{Lab19}. However, we take slightly different renormalisation constants compared to~\cite{Lab19}: instead of taking the constants built from the kernel $P^{(m)}_+$, we take those associated to $P^{(1)}_+$. Namely, in dimension $2$, we take $C_\eps := C_\eps^{(1)}$ and in dimension $3$, we take the three constants $c_\eps^{(1)}$, $c_\eps^{(1),1,1}$ and $c_\eps^{(1),1,2}$. This produces a limiting renormalised model that differs from the one in~\cite{Lab19} by \emph{finite} constants, as shown by the above asymptotics: this does not modify the final operator, but greatly simplify its definition (in particular, one does not need to deal with constants like $C^{(m)-(1)}$ as in~\cite{Lab19}).\\ 

Let us finally mention that the renormalised model satisfies for $d=2$
\[ \Pi_x^{(m),\eps}(\beta) \Xi \cI(\Xi)(x) = - C_\eps(\beta)\;,\]
and for $d=3$
\[ \Pi_x^{(m),\eps}(\beta) \Xi \cI(\Xi)(x) = - c_\eps^{(1)}(\beta)\;,\quad \Pi_x^{(m),\eps}(\beta) \Xi \cI(\Xi \cI(\Xi \cI(\Xi)))(x) = - c_\eps^{(1),1,1}(\beta) - c_\eps^{(1),1,2}(\beta)\;.\]

\section{Canonical model}\label{Appendix:Canonical}
The aim of this section is to construct a canonical admissible model $Z^{(m)}(V) := (\Pi^{(m)}_x(V), \Gamma^{(m)}_{xy}(V))$ associated to some potential function $V\in L^2_{\tiny\mbox{loc}}(\R^d)$ and to show that
\begin{equation}\label{Eq:BdCanonical}
\sup_{m\ge 1} \normm{Z^{(m)}(V);Z^{(m)}(0)}_Q < \infty\;.
\end{equation}
for all box $Q \subset \R^d$. If $V$ were smooth, then \cite[Prop. 8.27]{Hai14} would ensure that the model is admissible. Moreover, once the model is defined, \eqref{Eq:BdCanonical} essentially follows from Lemma \ref{lem:kernel} since the kernel $P^{(m)}_+$ is controlled by $G$ uniformly over all $m$. However, here $V$ is only in $L^2_{\tiny\mbox{loc}}(\R^d)$, which necessitates some adjustments in order to obtain the required analytical bounds. In the sequel, we fix $V \in L^2_{\tiny\mbox{loc}}(\R^d)$.

The set of symbols $\cT = \cU \cup \cF$ introduced in Subsection \ref{Subsec:RS} can be obtained through a recursive construction: let $\cU_0 = \{\un, X^k \}$, and for $n\ge 0$ we define recursively
\begin{align*}
\cF_n &:= \{\Xi \tau: \tau \in \cU_n\}\;,\\
\cU_{n+1} &:= \{\un, X^k\} \cup \{\cI(\tau): \tau \in \cF_{n}, |\tau|+2 < \gamma \}\;,
\end{align*}
for $\gamma = 2 - 4\kappa$. Subsequently we have $\cF = \bigcup_{n \ge 0} \cF_n$ and $\cU = \bigcup_{n\ge 0} \cU_n$. Note that all elements in $\cU$ are of positive homogeneity.

Given this recursive structure, we can define the model $\Pi^{(m)}$ in the following manner: for $m \ge 1$, $x,y \in \R^d$, we set
\[\Pi^{(m)}_x \un (y) = 1,\quad \Pi^{(m)}_x X^k(y) = (y - x)^k,\quad \Pi^{(m)}_x \Xi(y) = V(y) \]
and then recursively
\begin{align*}
\Pi^{(m)}_x (\tau \Xi)(y) &= (\Pi^{(m)}_x \tau)(y) \cdot V(y), \\
\Pi^{(m)}_x (\cI \bar{\tau})(y) &= \int P^{(m)}_+(y - z) \Pi^{(m)}_x \bar{\tau}(z) dz - \sum_{|k| < |\bar{\tau}| + 2} \frac{(y - x)^k}{k!} \int D^k P^{(m)}_+(x - z) \Pi^{(m)}_x \bar{\tau}(z) dz
\end{align*}
for all $\tau \in \cU$ and $\bar{\tau} \in \cF$. Since $V$ is a function, all these expressions are well-defined. 

We need the following estimate, that follows from standard arguments based on Lemma \ref{lem:kernel}. If $f_x$ is a function that satisfies
\[ |f_x(y)| \lesssim |x-y|^\zeta\;,\]
uniformly over all $y\in \R^d$ such that $|x-y| \le C$, then
\[ \left|\int P^{(m)}_+(y - z) V(z) f_x(z) dz - \sum_{|k| < |\tau| + 2} \frac{(y - x)^k}{k!} \int D^k P^{(m)}_+(x - z) V(z)f_x(z) dz\right| \lesssim |x-y|^{\zeta+2-\frac{d}{2}}\;,\]
uniformly over all $y\in\R^d$ such that $|x-y|\le C-1$ and over all $m\ge 1$.

Let us now prove recursively the required analytical bounds on $\Pi^{(m)}$ for some fixed box $Q$. Pick a large constant $C>0$. Suppose that for all $\tau \in \cU_n$, we have $|\Pi^{(m)}_x \tau (y)| \lesssim |y - x|^{|\tau|}$ uniformly over all $x \in Q$ and all $y\in\R^d$ such that $|x-y|\le C$ and over all $m\ge 1$. For any $\tau \in \cU_n$ we have, by the Cauchy-Schwarz inequality
$$ \big|\langle \Pi_x^{(m)} \Xi\tau,\varphi_x^\lambda\rangle\big| \lesssim \lambda^{|\tau|} \int |V(y)| |\varphi_x^\lambda(y)| dy \le \lambda^{|\tau|} \|V\, \un_{B(x,C)} \|_{L^2} \lambda^{-\frac{d}{2}} \lesssim \lambda^{|\tau\Xi|}\;,$$
uniformly over all $x \in Q$, all $m\ge 1$, all $\varphi\in \ccB^r$ and al $\lambda \in (0,C]$. Furthermore, by the estimate above
$$ \big|\Pi_x \cI(\Xi\tau)(y)\big| \le |x-y|^{|\tau|+2 - \frac{d}{2}}\;,$$
uniformly over all $x \in Q$ and all $y\in\R^d$ such that $|x-y|\le C-1$ and over all $m\ge 1$.

Since only finitely many iterations suffice to exhaust the whole set $\cU\cup\cF$, we deduce that
\begin{align*}
	\left| \crochet{\Pi^{(m)}_x \tau, \varphi^\lambda_x} \right| \lesssim \lambda^{|\tau|}
\end{align*}
uniformly over all $x\in Q$, all $m\ge 1$, all $\varphi\in \ccB^r$, all $\lambda \in (0,1]$ and all $\tau \in \cU\cup \cF$.

Regarding the construction of $\Gamma^{(m)}$, we argue that it is uniquely determined once $\Pi^{(m)}$ is specified on the negative levels of the regularity structure, see for example \cite[Thm. 2.10]{HW14}. Now that the model $Z^{(m)}(V)$ is defined with respect to $V \in L^2_{\tiny\mbox{loc}}$ and that the bound on $\normm{\Pi^{(m)}}_Q$ is independent of $m$, we can invoke \cite[Lem. 2.3]{Lab19} to conclude $\normm{Z^{(m)}(V); Z^{(m)}(0)}_Q$ is bounded by a constant independent of $m$, whence \eqref{Eq:BdCanonical} follows.

\section{Proof of Proposition \ref{prop:boites}}\label{sec:boites}

 The goal of this subsection is to prove Proposition \ref{prop:boites}, which is basically an extension of a result by G\"artner and K\"onig \cite{GK00} where they have considered the case of smooth bounded potential. We begin by proving a variation to their Proposition 1 in \cite{GK00}.

\begin{proposition}
For fixed $L > r > 0$ and any bounded smooth potential $V$, there exists a constant $K > 0$ such that
\[ \lambda(Q_L, V) \geq \min_{k \in \Z^d : |k|_\infty \leq \frac{L}{2r} + \frac{3}{4} } \lambda(rk + Q_{3r/2}, V) - \frac{K}{r^2}. \]
\label{prop:GK00}
\end{proposition}
\begin{proof}
The proof is built upon a specific choice of partition of unity: Let $\eta: \R^d \to [0, 1]$ be a smooth function supported in $Q_{3r/2}$ such that it gives $1$ on $Q_{r/2}$, $\sum_{k \in \Z^d} \eta_k^2(x) = 1$ and that $\sum_{k \in \Z^d} |\nabla \eta_k|^2(x) \leq K/r^2$ for all $x \in \R^d$, where $\eta_k(x) := \eta(r k + x)$. We will give a construction of such partition later in the proof.

Note first that we have the following variational formulation for the principal eigenvalue of the operator $-\Delta + V$ on a domain $D \subset \R^d$:
\[ \lambda(D, V) = \inf_{\substack{\psi \in C^\infty_c(D) \\ \|\psi\|_{L^2} = 1}} \int_{\R^d} |\nabla \psi|^2 + V \psi^2 =: \inf_{ \|\psi\|_{L^2} = 1} G^V(\psi). \]
Given the desired partition of unity $(\eta_k)$, we take $\psi \in C^\infty_c(Q_L)$ such that $\|\psi\|_{L^2} = 1$ and set $\psi_k = \eta_k \psi$. With the fact that $|\nabla \psi_k|^2 = \eta_k^2 |\nabla \psi|^2 + \psi^2 |\nabla \eta_k|^2 + \nabla (\eta_k^2) \cdot \nabla (\psi^2) /2$, it follows that $\sum_k |\nabla \psi_k|^2 = |\nabla \psi|^2 + \sum_k |\nabla \eta_k|^2 \psi^2$. Therefore
\begin{align*}
\sum_{k \in \Z^d} \|\psi_k\|_{L^2}^2 G^V(\psi_k/\|\psi_k\|_{L^2})
= \sum_{k \in \Z^d} \int_{\R^d} \left(|\nabla \psi_k|^2 + V \psi_k^2\right)
\leq G^{V}(\psi) + \frac{K}{r^2}
\end{align*}
where we have used the property $\sup_{x \in \R^d} \sum_{k} |\nabla \eta_k|^2(x) \leq K/r^2$ in the last inequality.
Since $\psi$ is supported in $Q_L$, the sum over $k$ is in fact a finite sum as we can restrict ourselves to those $k$'s such that $r|k|_\infty - 3r/4 < L/2$. Hence
\[G^{V} (\psi) + \frac{K}{r^2} \geq \sum_{|k|_\infty < \frac{L}{2r} + \frac{3}{4}} \|\psi_k\|_{L^2}^2 \min_{k \in \Z^d : |k|_\infty < \frac{L}{2r} + \frac{3}{4}  } \lambda(Q_{kr+3r/2}, V) = \min_{|k|_\infty < \frac{L}{2r} + \frac{3}{4} } \lambda(Q_{kr+3r/2}, V).\]
We then have our desired inequality by taking an infimum over $\psi$.

Finally, we finish this proof by constructing the function $\eta$ with desired properties. The $d$-dimensional construction can be first reduced to a $1$-dimensional one by setting $\eta(x) = \zeta(x_1) \dots \zeta(x_d)$ for $x = (x_1, \dots, x_d) \in \R^d$ with $\zeta$ being the $1$-dimensional version of $\eta$. On the other hand $\zeta$ can be constructed as follows. Let $\varphi(x) := c \int_{-\infty}^x e^{-1/(1 - u^2)} \un_{|u| < 1}$, where the constant $c$ is chosen so that $\varphi(x) = 1$ for $x \ge 1$. Note that $\varphi(x) = 0$ for $x\le -1$ and that $\varphi(x)+\varphi(-x) = 1$ for all $x\in\R$. One can also verify that $\sqrt{\varphi}$ is smooth.
Now set
\[\zeta(x) = \sqrt{\varphi(2(r + 2x)/r) \varphi(2(r - 2x)/r)}.\]
One can see that $\zeta(x) = 1$ if $|x| \le r/4$, $\zeta(x) = 0$ if $|x| > 3r/4$ and $\sum_k \zeta^2(rk + x) = 1$ for all $x$. Moreover, since the function $\varphi$ is independent of $r$, we have $\|\nabla \eta\|_{\infty} \lesssim 1/r$ with the proportionality constant depending only on the function $\varphi$, thus giving the bound $\sup_{x \in \R^d} \sum_{k} |\nabla \eta_k|^2(x) \leq K/r^2$ for some constant $K > 0$.
\end{proof}

\begin{proof}[Proof of Proposition \ref{prop:boites}]
To prove Proposition \ref{prop:boites}, we first consider the same assertions with $\xi$ replaced by the mollified and renormalized white noise $\xi_\eps$. In this case, the lower bound of (\ref{eq:boxes}) follows from Proposition \ref{prop:GK00}, while the remaining assertions are consequences of the variational formuation of eigenvalues
\[\lambda_n(D, V) := \inf_{\substack{F \sqsubset C^\infty_c(D)\\\mathrm{dim}(F) = n}} \sup_{\substack{\psi \in F\\\norm{\psi}_{L^2} = 1}} G^V(\psi)\]
(where the functional $G^V$ is defined in the proof of Lemma \ref{prop:GK00}) for any domain $D \subset \R^d$ and bounded smooth potential $V$.

Assertions being established for all smooth bounded $V$, now it remains to take $V_\eps = \beta\xi_\eps + C_\eps(\beta)$ and to pass to the limit. By Proposition \ref{Prop:Construct}, the eigenvalues of the renormalized Hamiltonian $\cH(Q_L,\beta \xi_{\eps_k} + C_{\eps_k}(\beta))$ converge almost surely to those of $\cH(Q_L,\beta \xi)$, which implies immediately the desired almost sure inequality (\ref{eq:boxes}).
\end{proof}

\endappendix

\bibliographystyle{Martin}
\bibliography{ref}

\begin{thebibliography}{{Gau}20}
\expandafter\ifx\csname url\endcsname\relax
  \def\url#1{\texttt{#1}}\fi
\expandafter\ifx\csname urlprefix\endcsname\relax\def\urlprefix{URL }\fi
\expandafter\ifx\csname href\endcsname\relax
  \def\href#1#2{#2}\fi
\expandafter\ifx\csname burlalt\endcsname\relax
  \def\burlalt#1#2{\href{#2}{\texttt{#1}}}\fi

\bibitem[AC15]{AllezChouk}
\textsc{R.~{Allez}} and \textsc{K.~{Chouk}}.
\newblock {The continuous Anderson hamiltonian in dimension two}.
\newblock \emph{arXiv e-prints}  arXiv:1511.02718.
\newblock \burlalt{arXiv:1511.02718}{http://arxiv.org/abs/1511.02718}.

\bibitem[Amb92]{Amb92}
\textsc{A.~Ambrosetti}.
\newblock Critical points and nonlinear variational problems.
\newblock \emph{M\'emoires de la Soci\'et\'e Math\'ematique de France} ,
  no.~49(1992).
\newblock \burlalt{doi:10.24033/msmf.362}{http://dx.doi.org/10.24033/msmf.362}.

\bibitem[Che14]{Chen}
\textsc{X.~Chen}.
\newblock Quenched asymptotics for {B}rownian motion in generalized {G}aussian
  potential.
\newblock \emph{Ann. Probab.} \textbf{42}, no.~2, (2014), 576--622.
\newblock \burlalt{doi:10.1214/12-AOP830}{http://dx.doi.org/10.1214/12-AOP830}.

\bibitem[Cv19]{CvZ19}
\textsc{K.~{Chouk}} and \textsc{W.~{van Zuijlen}}.
\newblock {Asymptotics of the eigenvalues of the Anderson Hamiltonian with
  white noise potential in two dimensions}.
\newblock \emph{arXiv e-prints}  arXiv:1907.01352.
\newblock \burlalt{arXiv:1907.01352}{http://arxiv.org/abs/1907.01352}.

\bibitem[CZ20]{CarZam}
\textsc{F.~Caravenna} and \textsc{L.~Zambotti}.
\newblock Hairer's reconstruction theorem without regularity structures.
\newblock \emph{EMS Surv. Math. Sci.} \textbf{7}, no.~2, (2020), 207--251.
\newblock \burlalt{doi:10.4171/emss/39}{http://dx.doi.org/10.4171/emss/39}.

\bibitem[DL20]{DL19}
\textsc{L.~Dumaz} and \textsc{C.~Labb\'{e}}.
\newblock Localization of the continuous {A}nderson {H}amiltonian in 1-{D}.
\newblock \emph{Probab. Theory Related Fields} \textbf{176}, no. 1-2, (2020),
  353--419.
\newblock
  \burlalt{doi:10.1007/s00440-019-00920-6}{http://dx.doi.org/10.1007/s00440-019-00920-6}.

\bibitem[DL21a]{DLCritical}
\textsc{L.~Dumaz} and \textsc{C.~Labb\'e}.
\newblock The delocalized phase of the {A}nderson {H}amiltonian in $1$-d.
\newblock \emph{arXiv e-prints} (2021).
\newblock \burlalt{arXiv:2102.05393}{http://arxiv.org/abs/2102.05393}.

\bibitem[DL21b]{DLCrossover}
\textsc{L.~{Dumaz}} and \textsc{C.~{Labb{\'e}}}.
\newblock {Localization crossover for the continuous Anderson Hamiltonian in
  $1$-d}.
\newblock \emph{arXiv e-prints} (2021).
\newblock \burlalt{arXiv:2102.09316}{http://arxiv.org/abs/2102.09316}.

\bibitem[Fra14]{Frank}
\textsc{R.~L. Frank}.
\newblock Ground states of semi-linear pdes.
\newblock \emph{Lecture notes} (2014).

\bibitem[{Gau}20]{Gaudreau}
\textsc{P.~Y. {Gaudreau Lamarre}}.
\newblock {Phase Transitions in Asymptotically Singular Anderson Hamiltonian
  and Parabolic Model}.
\newblock \emph{arXiv e-prints}  arXiv:2008.08116.
\newblock \burlalt{arXiv:2008.08116}{http://arxiv.org/abs/2008.08116}.

\bibitem[GIP15]{GIP15}
\textsc{M.~Gubinelli}, \textsc{P.~Imkeller}, and \textsc{N.~Perkowski}.
\newblock Paracontrolled distributions and singular {PDE}s.
\newblock \emph{Forum Math. Pi} \textbf{3}, (2015), e6, 75.
\newblock
  \burlalt{doi:10.1017/fmp.2015.2}{http://dx.doi.org/10.1017/fmp.2015.2}.

\bibitem[GK00]{GK00}
\textsc{J.~G{\"a}rtner} and \textsc{W.~K{\"o}nig}.
\newblock Moment asymptotics for the continuous parabolic anderson model.
\newblock \emph{Ann. Appl. Probab.} \textbf{10}, no.~1, (2000), 192--217.
\newblock
  \burlalt{doi:10.1214/aoap/1019737669}{http://dx.doi.org/10.1214/aoap/1019737669}.

\bibitem[GUZ20]{GUZ}
\textsc{M.~Gubinelli}, \textsc{B.~Ugurcan}, and \textsc{I.~Zachhuber}.
\newblock Semilinear evolution equations for the {A}nderson {H}amiltonian in
  two and three dimensions.
\newblock \emph{Stoch. Partial Differ. Equ. Anal. Comput.} \textbf{8}, no.~1,
  (2020), 82--149.
\newblock
  \burlalt{doi:10.1007/s40072-019-00143-9}{http://dx.doi.org/10.1007/s40072-019-00143-9}.

\bibitem[Hai14]{Hai14}
\textsc{M.~Hairer}.
\newblock A theory of regularity structures.
\newblock \emph{Inventiones mathematicae} \textbf{198}, no.~2, (2014),
  269--504.
\newblock
  \burlalt{doi:10.1007/s00222-014-0505-4}{http://dx.doi.org/10.1007/s00222-014-0505-4}.

\bibitem[Hel13]{Helffer}
\textsc{B.~Helffer}.
\newblock \emph{Spectral theory and its applications}, vol. 139 of
  \emph{Cambridge Studies in Advanced Mathematics}.
\newblock Cambridge University Press, Cambridge, 2013.

\bibitem[HP15]{HaiPar}
\textsc{M.~Hairer} and \textsc{E.~Pardoux}.
\newblock A {W}ong-{Z}akai theorem for stochastic {PDE}s.
\newblock \emph{J. Math. Soc. Japan} \textbf{67}, no.~4, (2015), 1551--1604.
\newblock
  \burlalt{doi:10.2969/jmsj/06741551}{http://dx.doi.org/10.2969/jmsj/06741551}.

\bibitem[HW14]{HW14}
\textsc{M.~Hairer} and \textsc{H.~Weber}.
\newblock Large deviations for white-noise driven, nonlinear stochastic pdes in
  two and three dimensions.
\newblock \emph{Annales de la facult{\'e} des sciences de Toulouse
  Math{\'e}matiques} \textbf{24}, no.~1, (2014), 55--92.
\newblock \burlalt{doi:10.5802/afst.1442}{http://dx.doi.org/10.5802/afst.1442}.

\bibitem[Lab19]{Lab19}
\textsc{C.~Labb{\'e}}.
\newblock {The continuous Anderson hamiltonian in $d\le 3$}.
\newblock \emph{{Journal of Functional Analysis}} \textbf{277}, no.~9(2019).
\newblock
  \burlalt{doi:10.1016/j.jfa.2019.05.027}{http://dx.doi.org/10.1016/j.jfa.2019.05.027}.

\bibitem[Lew10]{Lewin}
\textsc{M.~Lewin}.
\newblock {Describing lack of compactness in Sobolev spaces}, 2010.
\newblock \urlprefix\url{https://hal.archives-ouvertes.fr/hal-02450559}.
\newblock Lecture - Taken from unpublished lecture notes "Variational Methods
  in Quantum Mechanics" written for a course delivered at the University of
  Cergy-Pontoise in 2010.

\bibitem[Lew19]{Lewin-Spectral}
\textsc{M.~Lewin}.
\newblock {Th{\'e}orie spectrale \& m{\'e}canique quantique}, 2019.
\newblock \urlprefix\url{https://hal.archives-ouvertes.fr/cel-01935749}.

\bibitem[Mat20]{Matsuda}
\textsc{T.~Matsuda}.
\newblock Integrated density of states of the anderson hamiltonian with
  two-dimensional white noise (2020).
\newblock \burlalt{arXiv:2011.09180}{http://arxiv.org/abs/2011.09180}.

\bibitem[McK94]{McKean}
\textsc{H.~P. McKean}.
\newblock A limit law for the ground state of {H}ill's equation.
\newblock \emph{J. Statist. Phys.} \textbf{74}, no. 5-6, (1994), 1227--1232.
\newblock
  \burlalt{doi:10.1007/BF02188225}{http://dx.doi.org/10.1007/BF02188225}.

\bibitem[Mou20]{Mouzard}
\textsc{A.~Mouzard}.
\newblock Weyl law for the anderson hamiltonian on a two-dimensional manifold
  (2020).
\newblock \burlalt{arXiv:2009.03549}{http://arxiv.org/abs/2009.03549}.

\end{thebibliography}
\section*{Data Availability}
Data sharing not applicable to this article as no datasets were generated or analysed during the current study.

\end{document}